\documentclass[11pt]{amsart}

\usepackage{amsthm,amssymb}
\usepackage{MnSymbol} 
\usepackage{fullpage}
\usepackage[utf8]{inputenc}
\usepackage{graphicx} 
\usepackage{subcaption}
\usepackage{amssymb}
\usepackage{xcolor}
\usepackage{gensymb}
\usepackage{float}
\usepackage{soul,xcolor}
\usepackage[utf8]{inputenc}
\usepackage[T1]{fontenc}
\usepackage[english]{babel}
\usepackage{enumitem}
\usepackage{caption}

\newcommand{\sign}{\mathsf{sign}}

\newcommand{\wed}{\mathsf{Wedge}}

\newcommand{\sgeom}{\mathsf{SGeom}}
\newcommand{\wit}{\mathsf{Wit}}
\newcommand{\inc}{\mathsf{in}}

\renewcommand{\vec}[1]{\ensuremath{\mathbf{#1}}}

\usepackage{soul,xcolor}

\usepackage{graphicx}

\newcommand{\sd}{\mathbb{S}^d}
\newcommand{\stw}{\mathbb{S}^2}

\newcommand{\rtw}{\mathbb{R}^2}
\newcommand{\rth}{\mathbb{R}^3}

\newtheorem{proposition}{Proposition}
\newtheorem{theorem}{Theorem}
\newtheorem{lemma}{Lemma}
\newtheorem{corollary}{Corollary}
\newtheorem{question}{Question}

\newtheorem{remark}{Remark}

\begin{document}
\setstcolor{red}

\title[Strong geometry: knots]{Strong geometry: knots} 

\author[Baptiste Gros]{Baptiste Gros}
\address{IMAG, Univ.\ Montpellier, CNRS, Montpellier, France}
\email{baptiste.gros@umontpellier.fr}
\author[Jorge L. Ram\'irez Alfons\'in]{Jorge L. Ram\'irez Alfons\'in}
\address{IMAG, Univ.\ Montpellier, CNRS, Montpellier, France}
\email{jorge.ramirez-alfonsin@umontpellier.fr}



\begin{abstract} In this paper, we introduce the notion of {\em strong geometry}, a structure composed by both the chirotope of a set of points $X$ in the $d$-dimensional space and the {\em wedge} chirotope which is the specific {\em adjoint} chirotope induced by the hyperplanes spanned by $X$. We present various properties relating these two chirotopes, for instance,  by introducing the {\em witness} chirotope, we are able to give a formula expressing the wedge chirotope in terms of the usual chirotope.
With this on hand, we answer positively a strong geometry version of a question due to M. Las Vergnas about reconstructing polygonal knots via chirotopes. Moreover, we also show that {\em linear spatial graphs} are determined by their corresponding strong geometries.
\end{abstract}

\maketitle

\section{Introduction}\label{sec;intro}

Many problems in discrete and convex geometry are based on finite sets of points in the Euclidean space. The combinatorial properties of the underlying point set, rather than its metric properties, are enough to solve a large number of such problems. 
The codification of the `combinatorial properties' of a point set usually refers to a natural associated notion called {\em chirotope} (initially known as {\em combinatorial geometry}, with other equivalent names including {\em order type} \cite{GP} and, in our case, {\em oriented matroid}). The chirotope captures the relative positions of small tuples of points without taking distances into account \cite{BMS}. In the affine plane, one such piece of information we can formulate in terms of chirotopes is the intersection properties of the segments spanned by the initial points. However, even though the chirotope tells us which segments intersect, it does not give any information about the relative positions of those intersection (say with respect to the spanned lines). The knowledge of the latter might play a key role in the study of certain geometric problems. For instance, Figure \ref{fig:tran} illustrates two hexagons having the same chirotope induced by their vertices (they both have the same set of circuits or equivalently set of minimal Radon partitions). However, in the hexagon on the left the lines joining opposite vertices intersect at point $p$ while they do not on the right one. 

\begin{figure}[h]
\begin{center}
 \includegraphics[width=.46\textwidth]{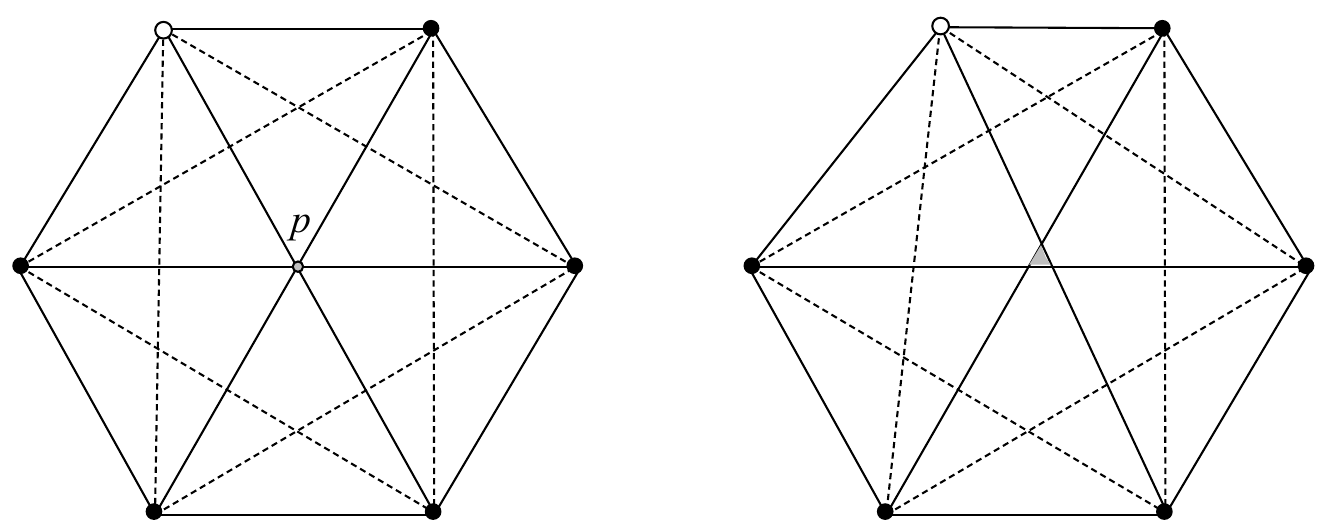}
 \caption{(Left) A regular hexagon where point $p$ (in gray) is the intersection of lines joining opposite vertices. (Right) The same regular hexagon in which a vertex (in white) has been slightly moved to the right. The lines joining opposite vertices form a triangle (in gray).}\label{fig:tran}
 \end{center}
\end{figure}

It turns out that $p$ is a {\em $0$-transversal} to the convex hull of the 4-sets (that is, $p$ intersects all 4-sets formed by the vertices). It can be easily checked that the hexagon on the right does not admit a $0$-transversal. In fact, such {\em discrete transversal} is not necessarily an
invariant of the order type, see \cite{CMMRA}. This is a typical situation in which the knowledge of the position of the intersection of two lines, with respect with the other spanned lines, is helpful (in this case for a transversality-type problem).  
\smallskip

Let $X$ be a vector configuration of $\mathbb {R}^d, d\ge 3$. We introduce the notion of {\em strong geometry} $\sgeom_{\mathsf{Lin}}(X)$ associated to $X$. This structure is composed of both the usual linear $M_{\mathsf{Lin}}(X)$ and the {\em wedge} $M_\wed(X)$ chirotopes associated to $X$. The wedge chirotope is a specific {\em adjoint} chirotope of $M_{\mathsf{Lin}}(X)$ associated to $X$ arising from the spanning vector configuration consisting of exactly one non-zero vector orthogonal to each hyperplane in $\mathbb {R}^d$ that is spanned by a subset of $X$, see  \cite[Section 3]{Cord1} (this canonical construction will be treated in detail below). This adjoint lives in the same dimension, but typically it has many more vectors than the original configuration. It is known that every rank 3 oriented matroid has an adjoint \cite{Good,Cord2,GP2}  but not every oriented matroid  has one, for instance, the V\'amos matroid has no oriented adjoint \cite{BW}. Although a realizable oriented matroid $M$ always admits an adjoint, the latter is not uniquely determined by $M$, it depends on the particular vector configuration representing $M$ (see below). We refer the reader to \cite{BK,Che} and \cite[Section 5.3]{BMS} for further discussions on adjoints for general oriented matroids. 
 \smallskip
 
Strong geometries give helpful information, for instance, they encode nicely the combinatorics of the cells of the arrangement of the hyperplanes spanned by a set of points. In Section \ref{sec:discuss}, we discuss different contexts closely connected to strong geometries.
\smallskip

In the Section \ref{sec:sg}, we treat in detail wedge chirotopes.  In rank $3$, that is, in the affine plane, we give a formula expressing the wedge chirotope in terms of the usual chirotope for some specific given configuration (see Eq.  \eqref{eq:delta} and Proposition \ref{prop:wit1}). Our approach allows us to get a `dimension reduction'  formula for any dimension. More precisely, we show that the chirotope of $M_\wed(X)$ determines the chirotope of $M_{\mathsf{Aff}}(X)$ up to orientation depending on the parity of the dimension (see Theorem \ref{thm:lin-wit} and Corollary \ref{cor:wedd}).  In order to show the latter, we introduce the {\em witness chirotopes} which play a central role for our purpose.
\smallskip

We then focus our attention  on the study of a first topological application of strong geometries in connection with geometric knots. A {\em knot} is an embeddings of $S^1$ into $\rth$ up to ambiant isotopy.  In order to avoid pathologies, it is common to study tame knots, in particular, {\em polygonal knots}, that is, closed, piecewise linear loops in $\rth$ with no self-intersections. Not much is known  about the interplay between the classic setting and {\em geometric knots}, which are polygonal knots with a fixed number of segments of variable length. Geometric knots are more rigid than smooth knots and are better suited to describe macromolecules (such as DNA) in polymer chemistry and molecular biology, but also far more resistant to investigation.
\smallskip

Let $X=(x_0, \dots, x_{n-1})$  be a $n$-tuple of points in $\rth$ in general position. Let $K_X$ be the polygonal knot defined by the segments $[x_i, x_{i+1}]$ (addition $\pmod n$).
Michel Las Vergnas \cite{LV} has put forward the following

\begin{question} Let $X= (x_0, \dots, x_{n-1})$ be a $n$-tuple of points in $\rth$ in general position. Is it true that $K_X$ only depends on the chirotope induced by $x_0,\dots ,x_{n-1}$ ?
\end{question}

By considering the affine version of strong geometries arising from a set of points $X$ in the space, we are able to give a positive answer to a strong geometry version of Las Vergnas' question, in Section \ref{sec:spatgraknot}. More conveniently, the result is stated in the following equivalent way.

\begin{theorem}\label{thm:main} Let $X=(x_0, \dots, x_{n-1})$ and $X'=(x'_0, \dots, x'_{n-1})$ be two $n$-tuples of points in $\rth$ in general position. Let $\sgeom_{\mathsf{Aff}}(X)$ and $\sgeom_{\mathsf{Aff}}(X')$ be the strong geometries associated to $X$ and $X'$ respectively. If $\sgeom_{\mathsf{Aff}}(X)$ is isomorphic to $\sgeom_{\mathsf{Aff}}(X')$ then $K_X$ is isotopic to $K_{X'}$.
\end{theorem}

Since the chirotope of $M_\wed(X)$ determines the chirotope of $M_{\mathsf{Lin}}(X)$ in dimension 4 (see Corollary \ref{cor:wed4}) then the following result is a straightforward consequence of Theorem \ref{thm:main}.

\begin{corollary} Let $X=(x_0, \dots, x_{n-1})$ and $X'=(x'_0, \dots, x'_{n-1})$ be two $n$-tuples of points in $\rth$ in general position. If $M_\wed(X)$ is isomorphic to $M_\wed(X')$ then $K_X$ is isotopic $K_{X'}$.
\end{corollary}

In order to obtain Theorem \ref{thm:main}, we prove that {\em knotoids} are determined by their Gauss diagram (see Lemma \ref{lem:gauss}).
The arguments for the latter are extended to show that {\em graphoids} are also determined by their Gauss diagram (see Lemma \ref{lem:gaussgraphdiag}).
This allows us to show, in Section \ref{sec:geomknot}, that {\em linear spatial graphs} are determined by their corresponding strong geometries as well (see Theorem \ref{thm:main2}).


\section{Oriented matroid preliminaries}

For general background in oriented matroid theory we refer the reader to the book \cite{BMS}. An {\em oriented matroid} $M=(E,\chi)$ of rank $r$ is a finite set $E=\{1,\dots ,n\}$ together with a function $\chi:E^r\rightarrow \{-1,0,1\}$, called {\em chirotope} verifying the following conditions

(CH0)  $\chi$ is not always zero,

(CH1)  $\chi$ is {\em alternating}, that is, for all $\{i_1,\dots ,i_r\}\subset E$ and all $\sigma\in \mathfrak{S}_r$, we have
$$\chi(i_{\sigma(1)},\dots ,i_{\sigma(r)})=\epsilon(\sigma)\chi(i_1,\dots ,i_r),$$

(CH3)  for all $\{i_1,\dots ,i_r\}, \{j_1,\dots ,j_r\}\subset E$ such that 
$$\chi(j_k,i_2,\dots ,i_r)\chi(j_1,\dots ,j_{k-1},i_1,j_{k+1},\dots ,j_r)\ge 0$$
for all $k$, then
$$\chi(i_1,\dots ,i_r) \chi(j_1,\dots ,j_r)\ge0.$$

Let $1\le d\le n$ be integers. To each configuration of vectors $X=(\vec x_1,\dots ,\vec x_n) \in (\mathbb{R}^d)^n$, we may associate an oriented matroid $M=(E,\chi)$ of rank $d$  by taking 
$$\chi({i_1},\dots ,{i_d})=\Delta(\vec x_{i_1},\dots , \vec x_{i_d})\in \{-1,0,1\}.$$

where $\Delta =\sign \circ \det$, that is, the sign of the determinant. 
\smallskip

$M$ is called {\em linear} or {\em vectorial} and it is denoted by $M_{\mathsf{Lin}}(X)$.

We may also associate an oriented matroid $M=(E,\chi)$ of rank $r(M)=d+1$ to a configuration of points $X=\{x_1,\dots ,x_n\}$ in the affine space $\mathbb{R}^d$ by taking  

$$\chi({i_1},\dots ,{i_d})=\Delta \left(\begin{array}{ccc} 1 & \cdots & 1 \\ x_{i_1} &\cdots &x_{i_d}\end{array}\right) .$$

$M$ is called {\em affine} and it is denoted by $M_{\mathsf{Aff}}(X)$.
\smallskip

We notice that we can pass from affine to linear matroids by considering the inclusion 
$$\inc:  \mathbb{R}^{d} \simeq \{1\} \times \mathbb{R}^d \subset \mathbb{R}^{d+1}.$$
 We may thus write $\vec x$ instead of $x$ to distinguish vectors from points. From now on, when we refer $\mathbb{R}^d$ as {\em affine space}, we mean that it is included in $\mathbb{R}^{d+1}$ as $\{1\} \times \mathbb{R}^{d}$.
\smallskip
 
Let us recall that a set $P=\{p_1,\dots ,p_{n}\}$ of $n\ge d+2$ points in $\mathbb{R}^d$ always admits a {\em Radon partition}, that is, a partition $P=A\sqcup B$ such that $conv(\{p_i \ |\  p_i\in A\})\cap conv(\{p_i \ |\  p_i\in B\})\neq\emptyset$. It is known that a circuit $C=C^+\cup C^-$ of an affine oriented matroid associated to a set of points $\{y_1,\dots ,y_{n}\}$ corresponds to a minimal Radon partition, that is,  $conv(\{y_i \ |\  i\in C^+\})\cap conv(\{y_i \ |\ i\in C^-\})\neq\emptyset$. If $(x_1, \dots, x_{d+1})$ is generic then a Radon partition of $\{x_1,\dots ,x_{d+2}\}$ admits both $x_{d+1}$ and $x_{d+2}$ in either $A$ or $B$ if and only if $\chi(x_1,\dots,x_d,x_{d+1})=-\chi(x_1,\dots,x_d,x_{d+2})$. Therefore, Radon's partitions can be recovered from the chirotope.


\section{Wedge matroids}\label{sec:sg}

We shall first discuss the construction of wedge oriented matroids for $d=2$ for a better comprehension. We then treat the general case $d\ge 3$.  

\subsection{Planar case} Let $X$ be a set of points in $\rtw$. As mentioned above, strong geometries intend to capture more geometric information than ordinary chirotopes, specifically, the relative positions of triples of lines generated by points in $X$.  As explained above, it will be more convenient to consider $X$ as included in $\mathbb{R}^3$.
Since we are only interested in computing signs of multi-linear expressions then we do not worry about the norm of the vectors in $\rth$. Therefore, the points in $X$ will be associated to vectors living in $\rth/(\mathbb{R_+})$ which is $\sd \cup\{0\}$.
\smallskip

Let $x$ and $y$ be two points in the affine space $\rtw\subset \rth$. Let $l_{x,y}$ be the equator arising from the intersection of the (oriented) plane $h_{x,y}$ spanned by $x$ and $y$ and $\mathbb{S}^2$. Let $\vec x$ and $\vec y$ be the vectors on $l_{x,y}$ arising from intersection of $\stw$ and the line segments $[0,x]$ and $[0,y]$ respectively. Equator $l_{x,y}$ is oriented from $\vec x$ to $\vec y$, see Figure \ref{fig:hyperplanes}. 

 \begin{figure}[H]
 \centering
\includegraphics[width=0.47\textwidth]{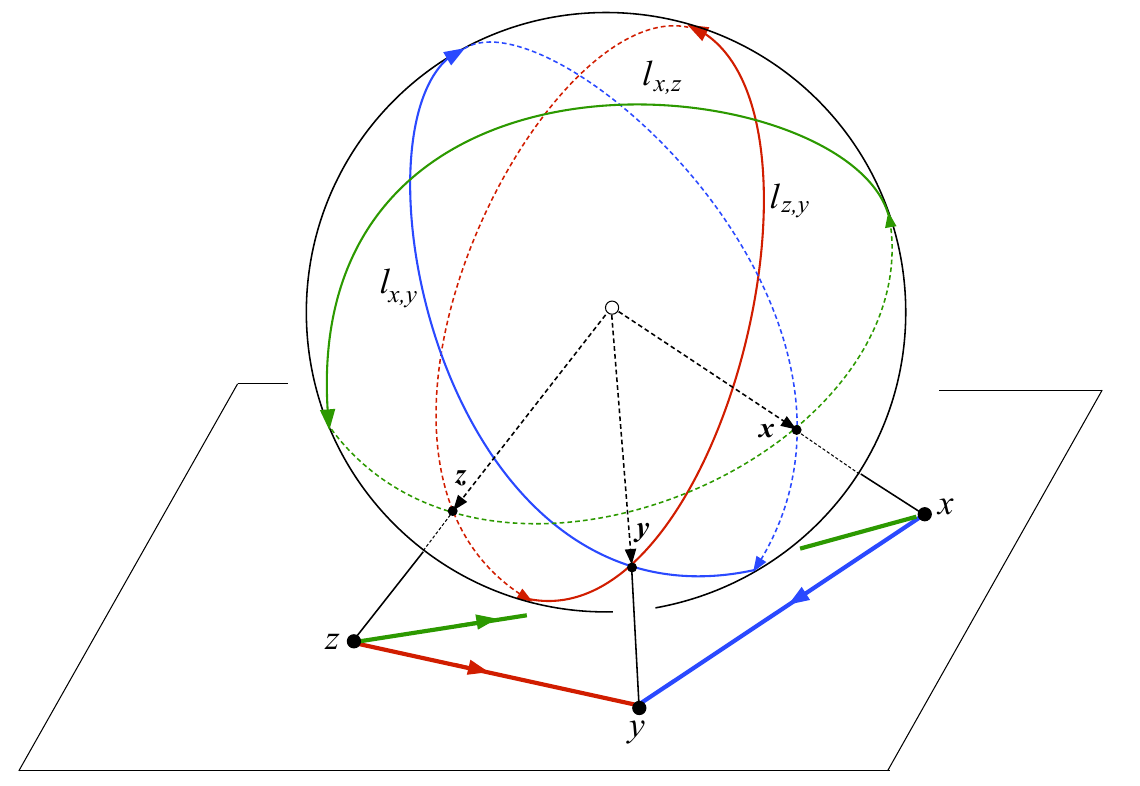}
    \caption{Unit vectors associated to points in the plane.} \label{fig:hyperplanes}
\end{figure}

If we write $\wedge$ for the image of $\times$ in $\rth/(\mathbb{R_+})$ then the equator $l_{x,y}$ is parametrized by its normal positive vector given by the product $\vec x\wedge \vec y$. 
In the degenerate case, we consider the 0 vector to be the equator between two equal or antipodal points, see Figure \ref{fig:hypercross}.

 \begin{figure}[H]
 \centering
\includegraphics[width=0.35\textwidth]{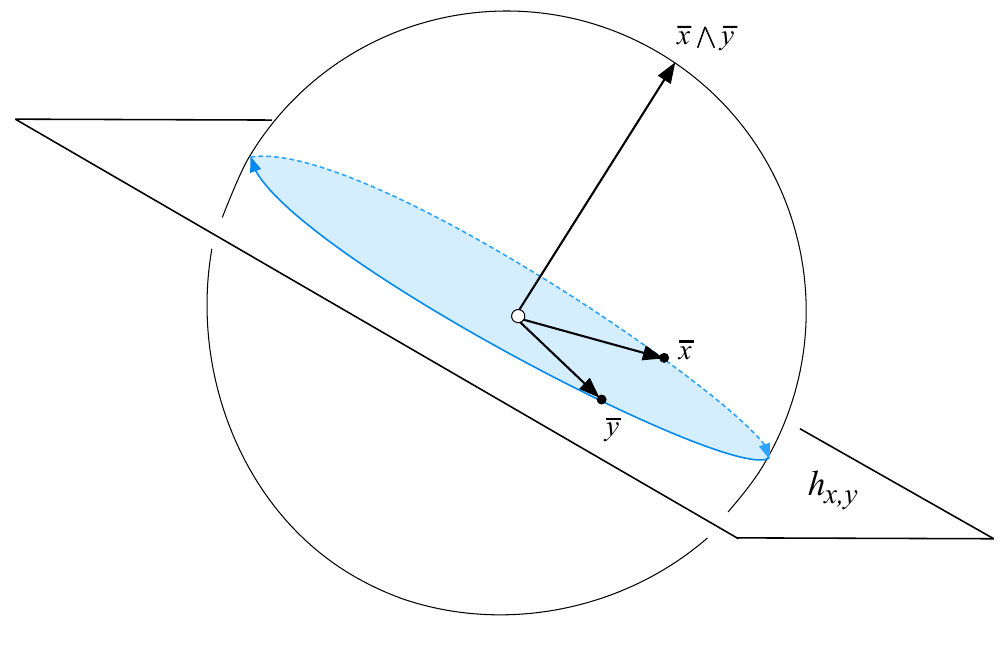}
    \caption{} \label{fig:hypercross}
\end{figure}

\begin{remark}\label{rem:pos} Let $l$ be an oriented equator of $\mathbb{S}^2$ with positive normal vector $\vec v_l$ and let $l^+$ (resp. $l^-$) be the positive (resp. negative) half-sphere delimited by $l$. Let $\vec p$ be a vector in $\stw$. Then,
$$\vec p\in\left\{\begin{array}{ll}
l^+ & \text{ if } \langle \vec v_l,\vec p\rangle > 0,\\
l^- & \text{ if } \langle \vec v_l,\vec p\rangle < 0,\\
l & \text{ if } \langle \vec v_l,\vec p\rangle=0.\\
\end{array}\right.$$
where $\langle \vec u,\vec v \rangle$ denotes the scalar product of vectors $\vec u$ and $\vec v$.
 \end{remark} 

Let $x,y$ and $z$ be points in the affine space. We commonly interpret $\Delta \left(\begin{pmatrix}
    1&1&1\\
    x&y&z
\end{pmatrix}\right)$ as indicating whether $z$ lies on the positive or negative half-space delimited by the line from $x$ to $y$.
Similarly, this sign is also given by either
\begin{equation}\label{eq:det1}
\Delta(\vec x,\vec y,\vec z)=\sign(\langle \vec x\wedge \vec y, \vec z\rangle),
\end{equation}
or the sign of the half-sphere delimited by the equator from $\vec x$ to $\vec y$ that contains $\vec z$

These observations serve as a guideline for a natural definition of the combinatorics of a tuple of oriented equators. Indeed, in affine geometry there are numerous ways that three oriented lines may met, see Figure \ref{fig:lines} while generic triplets of oriented equators can only be of two combinatorial types, see Figure \ref{fig:equat}.

 \begin{figure}[H]
 \centering
\includegraphics[width=0.7\textwidth]{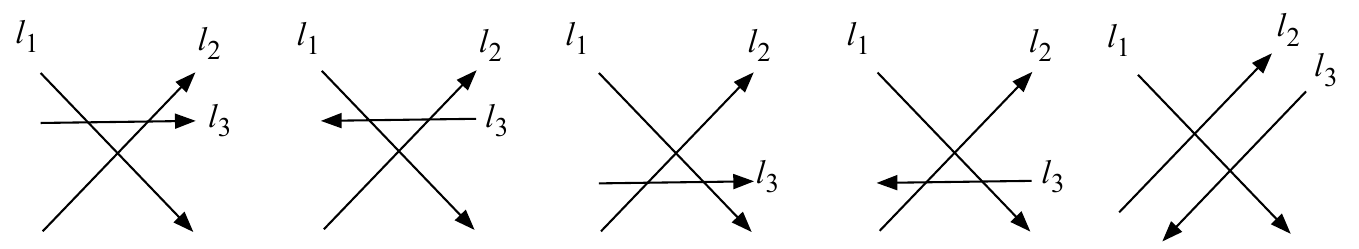}
    \caption{Some of the non-isotopic triplets of oriented lines.} \label{fig:lines}
\end{figure}

 \begin{figure}[H]
 \centering
\includegraphics[width=0.36\textwidth]{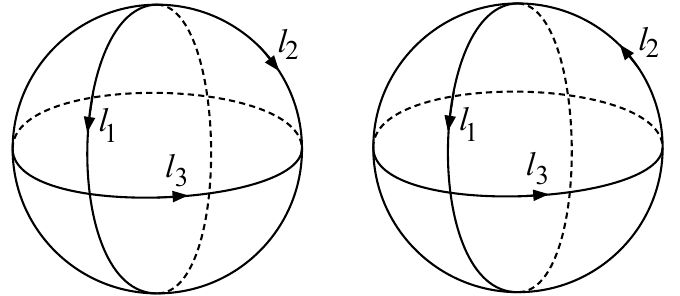}
    \caption{Positive and negative triplets of equators.} \label{fig:equat}
\end{figure}

Let $l_1$ and $l_2$ be two equators and let $\vec v_{1}$ and $\vec v_2$ be the normal vectors of the corresponding hyperplanes supporting $l_1$ and $l_2$ respectively. If $l_1$ and $l_2$ are {\em colinear}, i.e. equal or opposite, then the sign associated to any triple $(l_1, l_2, l)$ is always $0$. Otherwise, they meet at two antipodal points $p$ and $-p$, in this case, $\vec v_{1}\wedge \vec v_2$ is one of these points, we call it {\em positive intersection}, see Figure \ref{fig:inter}.

 \begin{figure}[H]
 \centering
\includegraphics[width=0.21\textwidth]{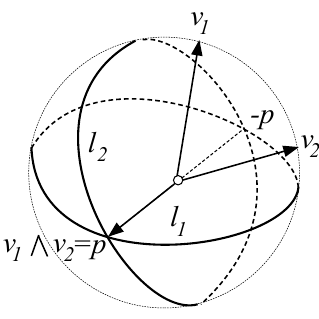}
    \caption{} \label{fig:inter}
\end{figure}

We thus have that a third equator $l_3$ admits either $p$ or $-p$ on its positive half-sphere (the other in the negative half-sphere). Analogously to the dual case above  \eqref{eq:det1}, we define the sign associated to a triple of equators $l_1, l_2$ and $l_3$ as

\begin{equation}\label{eq:det2}
\sign(l_1,l_2,l_3):=\sign( \langle \vec v_1\wedge \vec v_2, \vec v_3\rangle).
\end{equation}

This leads us to define a chirotope $\chi_{\Lambda}$ on the set of equators by

$$
\chi_{\Lambda}(i,j,k) := \sign(l_i,l_j,l_k).
$$

We call $\chi_\Lambda$ a {\em wedge} chirotope.  

Note that if $l_i = \vec x_i \wedge \vec y_i$,  $l_j = \vec x_j \wedge \vec y_j$ and $l_k = \vec x_k \wedge \vec y_k$, then we also have

\begin{equation}\label{eq:delta}
\chi_{\Lambda}(i,j,k) = \Delta\Big( (\vec x_i \wedge \vec y_i) \wedge (\vec x_j \wedge \vec y_j), \vec x_k, \vec y_k\Big).
\end{equation}

\subsection{General setting} We shall generalize the wedge chirotope to higher dimensions. To this end, we use the natural generalization of the cross product to higher dimension.


Let $d\ge 2$ and let $X= (\vec x_1,\dots ,\vec x_d)$ be a $d$-tuple of vectors in $\mathbb{R}^{d+1}$ (or in $\mathbb{S}^{d}$). If the family $(\vec x_1, \dots, \vec x_d)$ is independent then we define $\boldsymbol{\alpha}(X)$ to be the (norm 1) $(d+1)$-dimensional vector orthogonal to all $\vec x_i$'s (i.e., orthogonal to the hyperplane $h(X)$ spanned by the $\vec x_i$'s) and such that $\det(\vec x_1,\dots ,\vec x_d,\boldsymbol{\alpha}(X))>0$.  We call $\boldsymbol{\alpha}(X)$ an {\em hyperplane-vector}. If $(\vec x_1, \dots, \vec x_d$ is dependent then we set $\boldsymbol{\alpha}(X)=\vec{0}$ (in this case the hyperplane vector is {\em degenerate}).
\smallskip

Let $\mathcal{H}=\left\{h(I) \ |\  I\in [n]^d \right\}$. We may also suppose that the elements in $\mathcal{H}$ are ordered, that is, if $h_i=h(I)$ and $h_j=h(J)$  then $h_i<h_j$ if $I<J$ (in lexicographic order). 
\smallskip

Let $\Lambda(X)$ be the family of all hyperplane-vectors, that is,
$$ \Lambda(X)=\Big(\boldsymbol{\alpha}(X_I) \Big)_{I\in [n]^d } $$

For short, we may write $\Lambda(X)=\Big(\boldsymbol{\alpha}(I) \Big)_{I\in [n]^d }$.

We define the {\em wedge oriented matroid}, denoted by $M_\wed(X)$, as the oriented matroid $M_\mathsf{Lin}(\Lambda(X)) $, that is, the $(d+1)$-rank oriented matroid with base set $[n]^d$ and with chirotope function $\chi_{\Lambda}$ verifying

$$
\chi_{\Lambda}(I_1, \dots, I_{d+1}) = \Delta( \boldsymbol{\alpha}(I_1), \dots, \boldsymbol{\alpha}(I_{d+1}))
$$
for $I_1, \dots, I_{d+1} \in [n]^d $. 

Similarly as in the 2-dimensional case (i.e., rank 3), in the same flavor as Equation \eqref{eq:delta}, we have the following result which is a straight consequence of the definition of $\boldsymbol{\alpha}$.
\begin{lemma}\label{lem:delta}
    Let $X_1, \dots, X_d, Y$ be $(d+1)$  $d$-tuples of vectors in $\mathbb{R}^{d}$ and let $Y=(\vec y_1, \dots, \vec y_d)$. Then,
    
    $$\begin{array}{ll}
    \Delta\Big( \boldsymbol{\alpha}(X_1), \dots, \boldsymbol{\alpha}(X_d), \boldsymbol{\alpha}(Y)\Big) & = \sign \left( \left\langle \boldsymbol{\alpha}\big( \boldsymbol{\alpha}(X_1), \dots, \boldsymbol{\alpha}(X_d)\big), \boldsymbol{\alpha}(Y)\right\rangle \right)\\
    & =\Delta\Big(\boldsymbol{\alpha}\big( \boldsymbol{\alpha}(X_1), \dots,  \boldsymbol{\alpha}(X_d),\vec y_1, \dots, \vec y_d\big)\Big).
    \end{array}$$
\end{lemma}

We define the {\em strong geometry} associated to $X$, denoted by $\sgeom_{\mathsf{Lin}}(X)$, as the structure composed by the couple $(M_{\mathsf{Lin}}(X), M_\wed(X))$. 
We may also consider the affine version $\sgeom_{\mathsf{Aff}} (X)=(M_{\mathsf{Aff}}(X), M_\wed(\inc(X)))$. Strong geometry will determine helpful complementary geometric information on the set $X$. Moreover, the chirotopes $\chi$ and $\chi_{\Lambda}$ are closely related as it is highlighted in the following two propositions.


\begin{proposition}\label{prop:wit1}
Let $P=\{p_a, p_b, p_c, p_x, p_y\}$ be points in the affine space $\rtw\subset\rth$. Let $h_1, h_2$ and $h_3$ be the planes spanned by the couples of 3-dimensional vectors 
$I_1=(\vec p_a,\vec p_b), I_2=(\vec p_a,\vec p_c)$ and $I_3=(\vec p_x,\vec p_y)$ respectively. Recall that $\boldsymbol{\alpha}(I_1)=\vec p_a\wedge \vec p_b,  \boldsymbol{\alpha}(I_2)=\vec p_a\wedge \vec p_c$ and $\boldsymbol{\alpha}(I_3)=\vec p_x\wedge \vec p_y$. Then,
\begin{equation}\label{eq:chi2}
\chi_{\Lambda}(I_1,I_2,I_3) = \chi(a, b, c) \chi(a, x, y).
\end{equation}
\end{proposition}

\begin{proof} If $\chi(a, b, c)=0$ then the equality \eqref{eq:chi2} can be easily checked. Let us thus suppose that $\chi(a, b, c)\neq 0$, that is, $p_a, p_b$ and $p_c$ are points in generic position.
Since $\vec p_a$ is in both $l_1=\vec p_a\wedge \vec p_b$ and $l_2=\vec p_a\wedge \vec p_c$, then the positive intersection of equators $l_1$ and $l_2$ is either $p_a$ if $\chi(a,b,c) = +1$ or  $-p_a$ if $\chi(a,b,c) = -1$ (if $\chi(a,b,c) = 0$ then $\vec p_a\wedge \vec p_b$ and $\vec p_a\wedge \vec p_c$ are colinear). This is due to the fact there is only one isotopy class of positive triples of vectors. 
\smallskip

We have that \eqref{eq:chi2} can be thus stated as

$$(\vec p_a \wedge \vec p_b) \wedge (\vec p_a \wedge \vec p_c) = \chi(a, b,c) \ \vec p_a.$$

Combining this equality with 
 $$\chi_{\Lambda}(I_1, I_2, I_3) = \sign \Big(\langle (\vec p_a \wedge \vec p_b) \wedge (\vec p_a \wedge \vec p_c), (\vec p_x \wedge \vec p_y) \rangle \Big)$$
 the desired equality follows, see Figure \ref{fig:chi2}.
\end{proof}

 \begin{figure}[H]
 \centering
\includegraphics[width=0.5\textwidth]{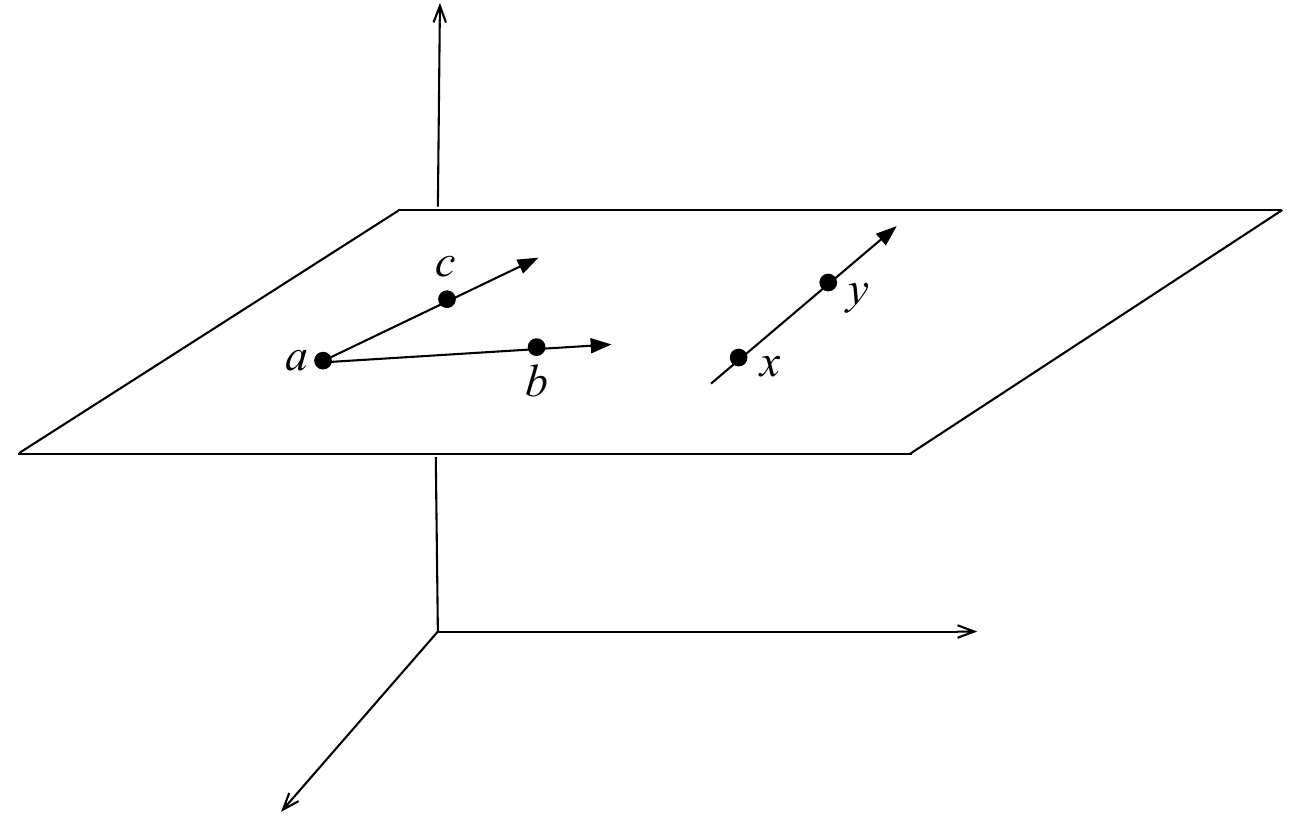}
    \caption{A visual explanation of Proposition \ref{prop:wit1} in the affine setting.} \label{fig:chi2}
\end{figure}

A consequence of Proposition \ref{prop:wit1} is the following

\begin{corollary}\label{cor:chir} Let $X$ and $X'$ be two $n$-tuples in $\mathbb{R}^3$. If $M_\wed(X)$
is isomorphic to $M_\wed(X')$ then the chirotopes of $M_{\mathsf{Lin}}(X)$  and $M_{\mathsf{Lin}}(X')$ are either equal or opposite. 
\end{corollary}

  \begin{proof} Suppose that $X,X'$ have the same wedge chirotope.
 By applying Proposition \ref{prop:wit1} to $(x,y) = (b,c)$, we obtain that $X$ and $X'$ agree on triples having sign equal to $0$.
 
 Now, again according to Proposition \ref{prop:wit1}, $X$ and $X'$ agree on whether two bases sharing some element have equal or opposite orientations. This implies that they also agree for any pair of bases $(x_{i_1},x_{i_2},x_{i_3})$ and $(x_{i_4},x_{i_5},x_{i_6})$, since, according to the matroid basis exchange property,
we can pass from $(x_{i_1},x_{i_2},x_{i_3})$ to $(x_{i_4},x_{i_5},x_{i_6})$ through$(x_{i_1}, x_{i_2}, x_{k})$ for some $ k \in \{i_4,i_5, i_6\}$.  This shows that the chirotopes are either equal or opposite.
      \end{proof}

We notice that, conversely, two opposite configurations have the same wedge chirotope.
 
We shall extend Proposition \ref{prop:wit1} to affine space $\mathbb{R}^3$.  To this end, we introduce the {\em witness chirotope} but, before doing so, let us first briefly discuss our motivation behind this new structure.
\smallskip

The main issue is that, instead of studying knots through their diagrams (orthogonal projections in the plane), we rather consider a projection of the knot radially to a sphere, say with center $\omega$. Therefore, given $x_1, \dots, x_n \in \mathbb{R}^3$, we are interested in the linear geometry of the  $x_1 - \omega, \dots, x_n -\omega \in \mathbb{R}^3$ (notice that the rank decreases by one). In this setting,
 the most relevant piece of information in the  knot sphere diagram will be given by  the strong geometry of such configuration (that is, a strong geometry witnessed by $\omega$). As we will see in Proposition \ref{prop:lin-wit}, the latter is induced by the strong geometry of the initial configuration since affine oriented hyperplanes $ \alpha(\vec \omega, \vec x_1, \vec y_1), \alpha(\vec \omega, \vec x_2, \vec y_2), \alpha(\vec \omega, \vec x_3, \vec y_3)$
meet either positively or negatively at $\omega$ depending on whether $\vec \omega$ sees the oriented lines $ (\vec x_1, \vec y_1), (\vec x_2, \vec y_2), (\vec x_3, \vec y_3)$ as a positive or negative triple. Notice that Proposition \ref{prop:wit1} is the rewording of the preceding remark in lower rank, although the strong geometry is hidden by the fact that, in rank $2$, it is the same as the usual chirotope.\\

 The above discussion naturally leads us to define the {\em witness} chirotope as follows. Let $\omega, x_{1},\dots ,x_{n}$ be a configuration of points in affine space $\mathbb{R}^d$, we note $\chi_{\omega}$ the linear chirotope of the vector configuration $(x_i - \omega)_{i \in [n]}$. Of course, since 
$$
\left| \begin{matrix}
    1& 1& \dots& 1\\
    \omega & x_{i_1} & \dots & x_{i_d}
\end{matrix}\right| = \left| \begin{matrix}
1 & 0 & \dots & 0\\
  \omega & x_{i_1}-\omega & \dots & x_{i_n}-\omega   
\end{matrix} \right| = | x_{i_1} - \omega, \dots, x_{i_n} - \omega|,
$$

then $\chi_{\omega}(\cdot, \dots, \cdot) = \chi(\omega, \cdot, \dots, \cdot)$, see figure \ref{fig:witn}.

 \begin{figure}[H]
 \centering
\includegraphics[width=0.8\textwidth]{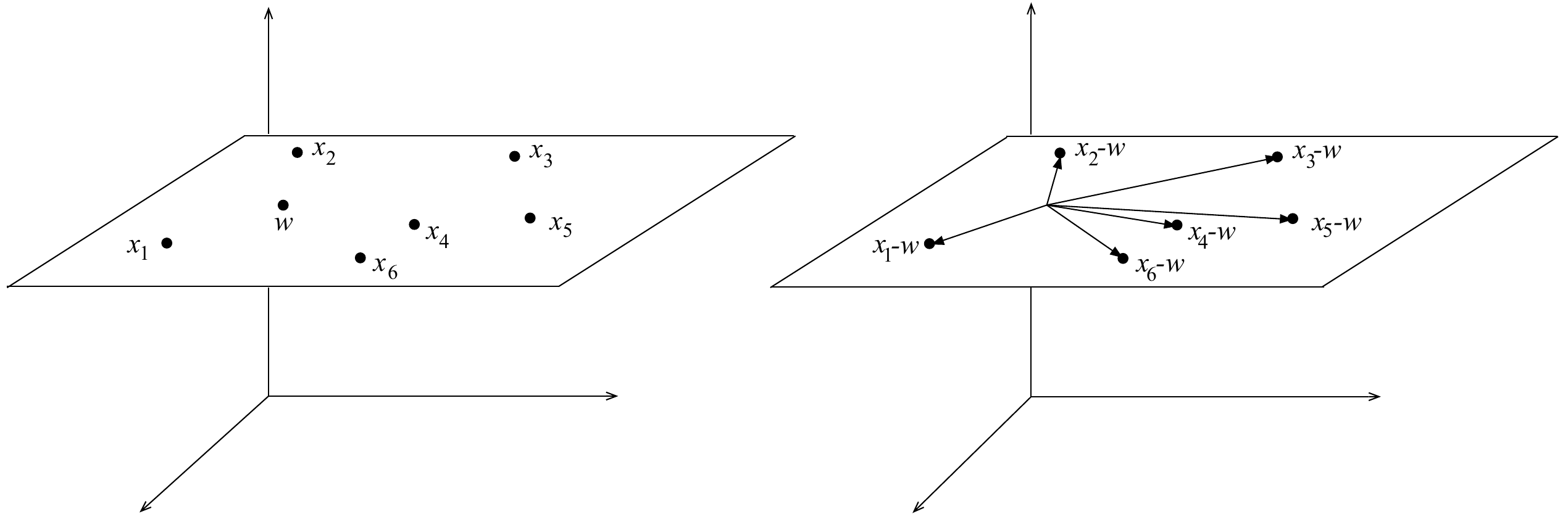}
    \caption{(Left) Affine oriented matroid of rank 3 associated to points $\omega, x_{1},\dots ,x_{6}$. (Right) Witness oriented matroid of rank 2 associated to vectors $\vec x_{1}-\boldsymbol{\omega},\dots ,\vec x_{6}-\boldsymbol{\omega}$.} \label{fig:witn}
\end{figure}

Since we want to study `witnessed strong geometry', it is also natural to define $\chi_{\omega, \wedge}$ on $(x_i) \in (\mathbb{R}^d)^n$ as the rank $d$, linear wedge chirotope of $(x_i - \omega)_{i \in  [n]}$ .\\

In order to present clearer and more proper proofs, it is convenient to have a second approach for the definition of a `witness chirotope', this time starting and ending in a linear space and, of course, coinciding with the above definition.\\

Since $x \mapsto x - \omega$ simply acts on $\{1 \} \times \mathbb{R}^d$ as the projection on $\{0 \} \times \mathbb{R}^d$ with kernel $<  \omega >$, the most obvious replacement is the orthogonal projection $\pi_{\omega^{\perp} }$ on $ \omega ^{\perp}$. In order to be consistent with the previous definition, we orient $\omega^{\perp}$ such that positive bases of $ \omega ^{\perp}$ are the $(b_1, \dots, b_{d-1})$ such that $(\omega, b_1, \dots, b_{d-1})$ is a positive basis of $\mathbb{R}^d$. We thus have that the linear equivalent of the above definition can be obtained by setting the witnessed strong geometry of $X$ witnessed by $\omega$ as the strong geometry of $\pi_{\omega^{\perp}}(X)$ (in $ \omega^{\perp} \simeq \mathbb{R}^{d-1}$).\\

Although both definitions seem to be the most natural ones, it is not immediate to convince one-self that they coincide. Indeed, if all the $\omega, x_i$ live in $\{1\} \times \mathbb{R}^3$, then the map $x \mapsto x-\omega$ acts on the $x_i$ as the projection on $\{0\} \times \mathbb{R}^3$ with kernel $\omega$, which is not necessarily the orthogonal projection (unless $\omega = e_1$ of course)\\

Let us check that both projections of the point configurations do have the same strong geometry.

\begin{lemma}\label{lem:equiv}
    Let $H_1, H_2$ be two hyperplanes of $\mathbb{R}^d$ not containing $\omega \neq 0$, and let $\pi_1, \pi_2$  the projections with kernel $< \omega>$ and respective images $H_1, H_2$. Suppose further that the $H_i$ are oriented such that a positive base $(b_1, \dots, b_{d-1})$ of $H_i$ induces a positive base $(\omega, b_1, \dots, b_{d-1})$ of $\mathbb{R}^d$. Then for any $X \in (\mathbb{R}^d)^n$, $\pi_1(X)$ and $\pi_2(X)$ have the same strong geometry.
\end{lemma}
\begin{proof}
    We first notice that the $\pi_{i}(X)$ only differ by a positive isomorphism. Indeed the composition $H_1 \hookrightarrow \mathbb{R}^d \rightarrow H_2$ of the canonical inclusion and the projection $\pi_2$ is a linear map between vector spaces of same dimensions, with trivial kernel, that is, an isomorphism. It also sends a basis $(b_1, \dots, b_{d-1})$ to some $(b_1 - \lambda_1 \omega, \dots, b_{d-1} - \lambda_{d-1} \omega)$, and since $\left| \omega, b_1 - \lambda_1 \omega, \dots, b_{d-1} - \lambda_{d-1} \omega \right| = \left| \omega, b_1, \dots, b_{d-1} \right|$, it is a positive isomorphism. We claim that a positive isomorphism preserve the strong geometry. \\
    
 Firstly, notice that for a family $(h_i)_{i\in I}$ of hyperplanes and $(v_i)_{i\in I}$ its family of normal vectors,  a vector $v$ is orthogonal to all $v_i$ if and only if it belongs to all $h_i$. But a $((d-1) \times (d-2) )$-tuple $(x_{i,j})$ verifies $\Delta\Big(\alpha (x_{1,1}, \dots,x_{1,d-2}), \dots, \alpha(x_{d-1, 1} ,\dots,  x_{d-1, d-2})\Big) = 0$  if and only if the $\alpha(x_{i, 1}, \dots, x_{i,d-2})$ are linearly dependent, or equivalently, if there is a non-zero vector orthogonal to all of them. Therefore, $\Delta\Big(\alpha (x_{1,1}, \dots,x_{1,d-2}), \dots, \alpha(x_{d-1, 1} , \dots, x_{d-1, d-2})\Big) = 0$ if and only if $\mbox{dim} \left( \bigcap_{i} < x_{i,1}, \dots, x_{i,d-2}>\right) \neq 0$. Now acting on the $x_{i,j}$ by isomorphism does not change this dimension, therefore
    $$\Delta\Big(\alpha (x_{1,1}, \dots,x_{1,d-2}), \dots, \alpha(x_{d-1, 1} ,\dots, x_{d-1, d-2})\Big) = 0$$ if and only if 
    $$\Delta\Big(\alpha (\phi(x_{1,1}), \dots,\phi(x_{1,d-2})), \dots, \alpha(\phi(x_{d-1, 1}) , \dots,\phi(x_{d-1, d-2}))\Big) = 0$$ for any isomorphism $\phi$. \\

    Take now a positive isomorphism $\phi$. Since the group of positive isomorphisms is path connected, we have a continuous path $\phi_t$ of (positive) isomorphisms such that $\phi_0 = \mbox{id}$, $\phi_1 = \phi$. If we had, say
    $$\Delta\Big(\alpha (x_{1,1}, \dots,x_{1,d-2}), \dots, \alpha(x_{d-1, 1} ,\dots, x_{d-1, d-2})\Big) >  0$$ and 
    $$\Delta\Big(\alpha (\phi(x_{1,1}), \dots,\phi(x_{1,d-2})), \dots, \alpha(\phi(x_{d-1, 1}) , \dots, \phi(x_{d-1, d-2}))\Big) < 0$$ on some $(x_{i,j})$, then there would be some $t \in (0,1)$ such that 
    $$\Delta\Big(\alpha (\phi_t(x_{1,1}), \dots,\phi_t(x_{1,d-2})), \dots, \alpha(\phi_t(x_{d-1, 1}) , \dots, \phi_t(x_{d-1, d-2}))\Big) < 0,$$ but, we just saw that this cannot happen. Hence, 
   {\small $$\Delta\Big(\alpha (x_{1,1}, \dots,x_{1,d-2}), \dots, \alpha(x_{d-1, 1} ,\dots, x_{d-1, d-2})\Big) = \Delta\Big(\alpha (\phi(x_{1,1}), \dots,\phi(x_{1,d-2})), \dots, \alpha(\phi(x_{d-1, 1}) , \dots,\phi(x_{d-1, d-2}))\Big).$$} 
 Implying that positive isomorphisms preserve the strong geometry, as desired.
\end{proof}

We have thus checked that the natural definition of witnessed strong geometry from affine space to linear space coincides with the definition from linear space to linear space.
A rigorous formulation of the latter is given by the following generalized version of Proposition \ref{prop:wit1}.



\begin{theorem}\label{thm:lin-wit}
    Let $\omega, x_{1,1}, \dots, x_{1, d-2}, \dots, x_{d-1, d-2}, y_1, \dots, y_{d-1} \in \mathbb{R}^d$, with $\omega \neq 0$. Let us denote by $\tilde{\alpha}$ and $ \tilde{\Delta}$ the $\alpha$ and $\Delta$ functions associated with the oriented hyperplane $\omega^{\perp}$ . Then, we have that

    \begin{align*}
&\Delta\Big( \alpha(\omega, x_{1,1}, \dots, x_{1,d-2}), \dots, \alpha(\omega, x_{d-1, 1}, \dots, x_{d-1, d-2}), \alpha(y_1, \dots, y_{d-1}) \Big) \\
= (-1)^{d+1} &\tilde{\Delta} \Big(\tilde{\alpha}(x_{1,1}, \dots, x_{1, d-2}), \dots, \tilde{\alpha}(x_{d-1, 1}, \dots, x_{d-1, d-2}) \Big) \Delta(\omega, y_1, \dots, y_{d-1})
    \end{align*}

\end{theorem}


\begin{proof}
    
  The proof runs similarly as  the one for Proposition \ref{prop:wit1}. The point is that the left hand term is the relative position between the hyperplane represented by $\alpha(y_1, \dots, y_{d-1}) $ and the positive intersection $\alpha(v_1, \dots, v_{d-1})$ of the hyperplanes represented by the $v_i = \alpha(\omega, x_{i, 1}, \dots, x_{i,d-2})$. The proof consists in showing that the latter is precisely $$(-1)^{d+1} \tilde{\Delta} \Big(\tilde{\alpha}(x_{1,1}, \dots, x_{1, d-2}), \dots, \tilde{\alpha}(x_{d-1, 1}, \dots, x_{d-1, d-2}) \Big) \omega.$$

  Since the hyperplanes represented by the $\alpha(\omega, x_{i, 1}, \dots, x_{i,d-2})$ all contain $\omega$, their positive intersection $\alpha(v_1, \dots, v_{d-1})$ is some $\eta \omega$, with $\eta \in \{-1, +1\}$ (unless it is trivial, in which case the proof is straightforward). By definition, 
  $$
 \Delta\Big( \alpha(\omega, x_{1,1}, \dots, x_{1,d-2}), \dots, \alpha(\omega, x_{d-1, 1}, \dots, x_{d-1, d-2}), \eta \omega\Big) =1
  $$

  so we simply want to compute

  $$
\eta = \Delta\Big( \alpha(\omega, x_{1,1}, \dots, x_{1,d-2}), \dots, \alpha(\omega, x_{d-1, 1}, \dots, x_{d-1, d-2}), \omega\Big) .
  $$

    Note that $\alpha(\omega, x_{i, 1}, \dots, x_{i,d-2}) = \tilde{\alpha}(p(x_{i,1}), \dots, p(x_{i, d-2}))$ where $p$ denotes the orthogonal projection on $\omega^{\perp}$ (because a vector orthogonal to $\omega$ is orthogonal to some $x$ if and only if it is orthogonal to $p(x)$, and because the choice of convention for the orientation of $\omega^{\perp}$ was made to have no "$-$" sign appear here).\\

    Therefore, we have 

    \begin{align*}
\eta &= \Delta\Big( \tilde{\alpha}(p( x_{1,1}), \dots, p(x_{1,d-2})), \dots, \tilde{\alpha}(p( x_{d-1, 1}), \dots, p(x_{d-1, d-2})), \omega\Big) \\
&= (-1)^{d+1} \Delta\Big( \omega, \tilde{\alpha}(p( x_{1,1}), \dots, p(x_{1,d-2})), \dots, \tilde{\alpha}(p( x_{d-1, 1}), \dots, p(x_{d-1, d-2}))\Big)\\
&= (-1)^{d+1} \tilde{\Delta} \Big(\tilde{\alpha}(x_{1,1}, \dots, x_{1, d-2}), \dots, \tilde{\alpha}(x_{d-1, 1}, \dots, x_{d-1, d-2}) \Big)
    \end{align*}

the last equality is by definition of $\tilde{\Delta}$. The desired equality follows.
\end{proof}

Let us rephrase Theorem \ref{thm:lin-wit} in terms of abstract chirotope. We restrict ourselves to the rank $4$ case and state it in terms of projection from affine space (handy for the knot's application).

\begin{proposition}\label{prop:wit2}  Let $n\ge 1$ be an integer and let $X=(\omega = x_0,x_1,\dots ,x_n)$ be a tuple of points in $\mathbb{R}^3$.

Let $i_0=0$, $i_1, \dots, i_9 \in [n]^9$ and let $I_1 = (i_0, i_1, i_2)$, $I_2 = (i_0, i_3, i_4)$, $I_3 = (i_0, i_5, i_6)$ and $J = (i_7, i_8, i_9)$. Let $a_1 = (i_1, i_2), a_2 = (i_3, i_4), a_3 = (i_5,i_6)$. Then, 
$$
\chi_{\wedge}(I_1, I_2, I_3, J) =  -\chi_{\Lambda, \omega}(a_1,a_2,a_3)  \chi(i_0,i_7, i_8, i_9)
$$
where $\chi_{\Lambda}, \chi$ and $\chi_{\wedge, \omega}$ are the chirotopes of the rank 4 matroids $M_\wed(\inc(X))$ and $M_{\mathsf{Lin}}(\inc(X)) = M_{\mathsf{Aff}}(X)$  and the rank 3 matroid $M_{\mathsf{Wit}_{\omega}}(X)$ respectively.\\
Here, $M_{\wit_{\omega}}(X)$ is the witness oriented matroid of the configuration of points in the sphere centered at $\omega$ obtained by radial projection of $X$.
\end{proposition}

\begin{proof}
    Since the norms of the vectors are not taken into account in a linear matroid and they can be projected on the sphere then the desired equality is simply expressing Theorem \ref{thm:lin-wit} through one of the equivalent definitions of wedge chirotope (by using Lemma \ref{lem:equiv}).
\end{proof}

 \begin{figure}[H]
 \centering
\includegraphics[width=0.38\textwidth]{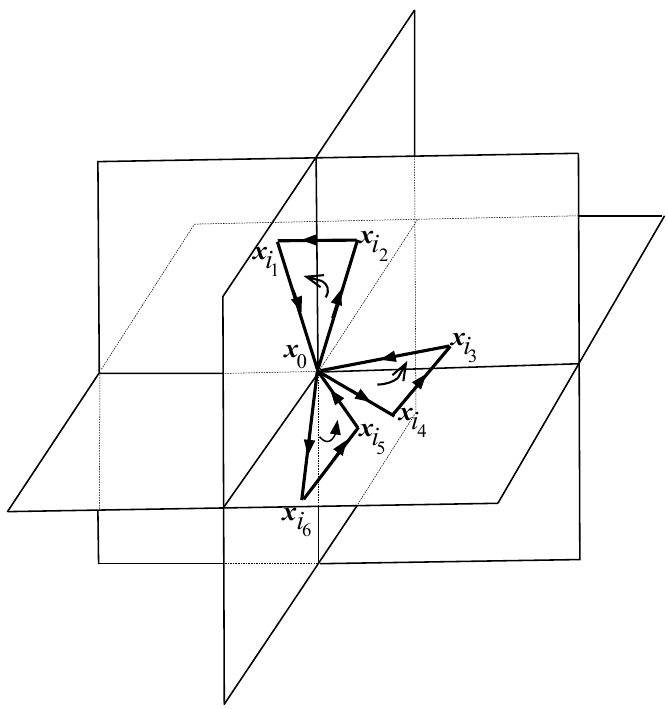}
    \caption{This figure sums up the situation of Proposition \ref{prop:wit2}. The three hyperplanes meet positively or negatively at $x_0$ depending on the sign of the triplet of lines witnessed by $x_0$.} \label{fig:witnhyp}
\end{figure}

\begin{corollary}\label{cor:wed4}
Let $X$ and $X'$ be two $n$-tuples in $\mathbb{R}^4$. If $M_\wed(X)$ and $M_\wed (X')$ are isomorphic then $M_{\mathsf{Lin}}(X)$  and $M_{\mathsf{Lin}}(X')$ also are. 
\end{corollary}

\begin{proof}

According to Proposition \ref{prop:wit1}, $\chi(a\wedge b, a \wedge c,  b\wedge c) = \chi(a,b,c)^2$. By plugging this into the equality of Proposition \ref{prop:wit2}, we obtain 

\begin{align*}
\chi \big((\alpha(x,a,b), \alpha(x,a,c), \alpha(x,b,c), \alpha(a,b,c)\big) &= \chi_x(a,b,c)^2 \chi(x, a, b,c)\\
&= \chi(x,a,b,c)^2 \chi(x,a,b,c)\\
&=\chi(x,a,b,c).
\end{align*}   
\end{proof}

Corollaries \ref{cor:chir} and \ref{cor:wed4} can be extended to higher dimension.

\begin{corollary}\label{cor:wedd} Let $X$ and $X'$ be two $n$-tuples in $\mathbb{R}^d$. If $M_\wed(X)$
is isomorphic to $M_\wed(X')$ then the chirotopes of $M_{\mathsf{Lin}}(X)$  and $M_{\mathsf{Lin}}(X')$ are equal if $d$ is even and either 
equal or opposite if $d$ is odd.  
\end{corollary} 

The proof is also a consequence of Theorem \ref{thm:lin-wit} (analogue to the lower rank cases with some extra tedious indices manipulation).


\section{Geometric knots}\label{sec:geomknot}

Let us first recall a simple way to encode a knot diagram that enable to reconstruct an equivalent diagram from it. 

\subsection{Gauss code} A {\em Gauss code} of an oriented knot diagram with $n$ crossings is constructed as follows. Label all the crossings of the diagram from $1$ to $n$ (in an arbitrary manner). 
The Gauss code is derived by walking the knot, starting at point $P$ of the diagram (picked arbitrarily  and other than a crossing). As we follow the diagram, we record the crossings we encounter, by writing down the labels preceded with an `O' or `U' to indicate whether the curve goes {\em Over} or {\em Under} strand, see Figure \ref{fig:codeG}.

 \begin{figure}[H]
 \centering
\includegraphics[width=0.2\textwidth]{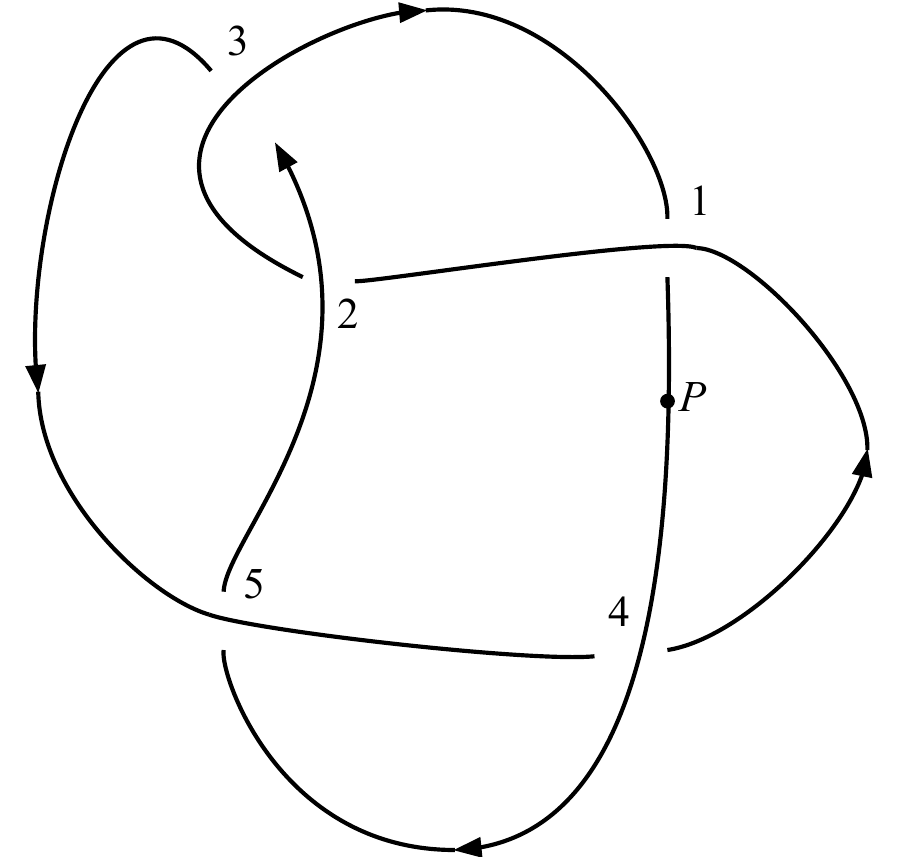}
    \caption{Knot $5_2$ with Gauss code \small{O4 U5 O2  U3  O5 U4  O1  U2  O3 U1}.} \label{fig:codeG}
\end{figure}

In general, the Gauss code for a knot diagram cannot be used to reconstruct an equivalent diagram, but an extension of it will make the reconstruction possible. The {\em extended Gauss code} is a minor revision of Gauss code : as we encounter a given crossing, we recorded it as above and we also assign a sign depending on the handedness of the crossing. If it is {\em right-handed}, we assign positive; if it is {\em left-handed}, negative, see Figure \ref{fig:cross}.

 \begin{figure}[H]
 \centering
\includegraphics[width=0.3\textwidth]{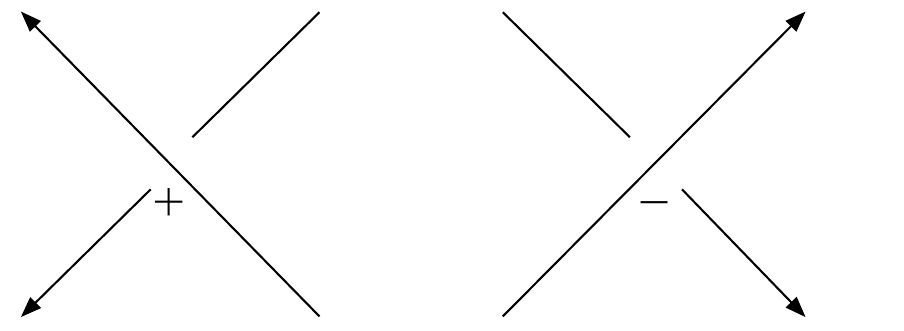}
    \caption{Positive and negative oriented crossings.} \label{fig:cross}
\end{figure}

The extended Gauss code of the example illustrated in Figure \ref{fig:codeG} is given by 
$$\hbox{$O\text{}_+4 \ \text{U}_+5 \ \text{O}_+2\  \text{U}_+3\  \text{O}_+5\ \text{U}_+4\ \text{O}_-1\  \text{U}_+2\  \text{O}_+3\ \text{U}_-1$. }$$

A diagrammatic representation of an extended Gauss code is given by a  {\em Gauss diagram} constructed as follows. Take an oriented circle with a base point chosen on the circle. Walk along the circle marking it with the labels for the crossings in the order of the Gauss code. Now draw chords between the points on the circle that have the same label. Orient each chord from the over crossing to the under crossing in the Gauss code.  Mark each chord with $+$ or $-$ according to the sign of the corresponding crossing in the Gauss code. An example of the Gauss diagram for the knot $6_3$ is given in Figure \ref{fig:gauss}.

 \begin{figure}[H]
 \centering
\includegraphics[width=0.42\textwidth]{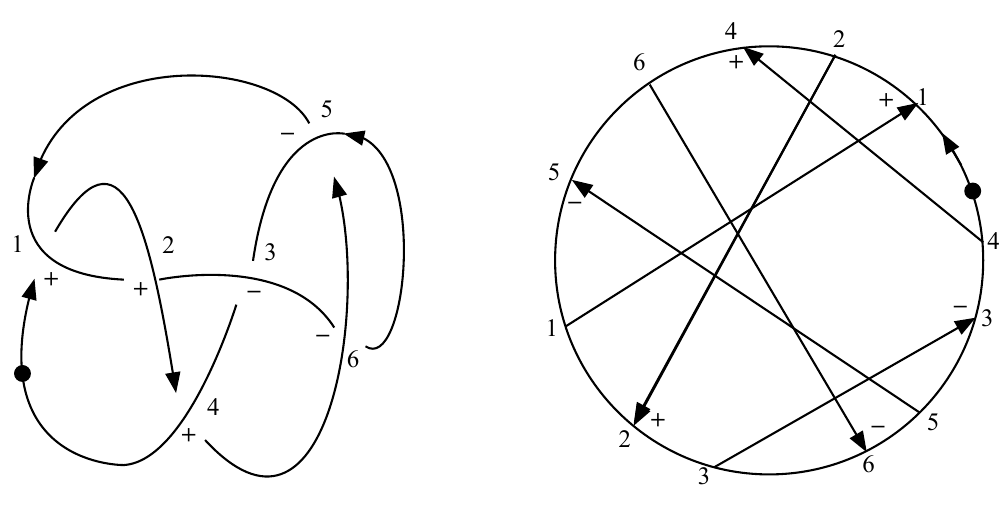}
    \caption{(Left) $6_3$ with Gauss code  \small{$\text{U}_+1\ \text{O}_+2\ \text{U}_+4\  \text{O}_-6\  \text{U}_-5\ \text{O}_+1\ \text{U}_+2\  \text{O}_-3\  \text{U}_-6\ \text{O}_-5\ \text{U}_-3\ \text{O}_+4$} (Right) Corresponding Gauss diagram.} \label{fig:gauss}
\end{figure}

It is known that a knot diagram on the sphere can be recovered uniquely (up to isotopy) from its Gauss diagram which can thus be considered as an alternative way to present knots. Unfortunately, not every picture which looks like a Gauss diagram is indeed a Gauss diagram of some knot. This is not easy to recognize, see \cite{CE}.

A Gauss diagram can also be represented by an oriented line where the ends are identified and having oriented signed arcs corresponding to the chords, see en Figure \ref{fig:gaussline}.

 \begin{figure}[H]
 \centering
\includegraphics[width=0.49\textwidth]{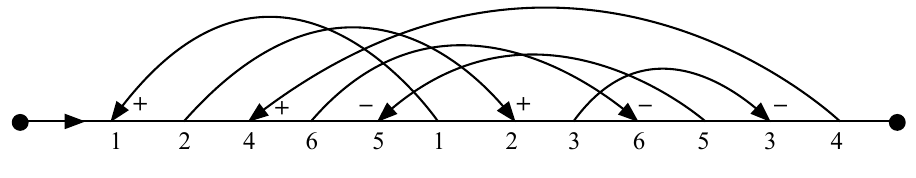}
    \caption{Line representation of the Gauss diagram illustrated in Figure \ref{fig:gauss}.} \label{fig:gaussline}
\end{figure}



\subsection{Knotoids}
{\em Knotoids} are variants of knots that were first introduced by Turaev  in \cite{Tur}. The idea is to consider diagrams of segments with endpoints. Contrarily to knots, they are only defined as equivalence classes of diagrams, since all embedding of the interval in $\rth$ are ambiant isotopic. More precisely, a {\em knotoid} is an equivalence class of {\em knotoid diagrams}. A {\em knotoid diagram} is the image of an application $[0,1] \to \mathbb{R}^2$ that satisfies the usual properties of a knot diagram, that is, each point has at most two preimages, called crossings, and there is a finite number of such crossings; besides, crossings happen between two transversal strands. We further suppose that the endpoints do not lie on crossings. The equivalence relationship considered here is the one generated by planar isotopy and the usual Reidemeister moves ; endpoints are not allowed to cross strands.

 \begin{figure}[H]
 \centering
\includegraphics[width=0.22\textwidth]{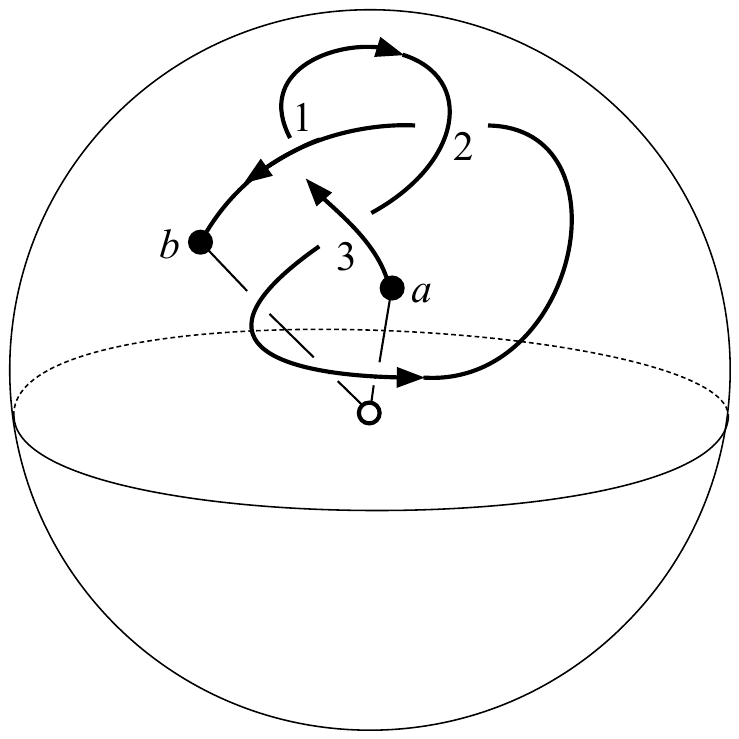}
    \caption{ The knotoid called {\em line eight-figure}.} \label{fig:linediag}
\end{figure}

Similarly as we do for knots, we may associate to a knotoid diagram its Gauss diagram.
Take an oriented segment, say $S$, with an initial and end points. Walk along $S$ marking it with the labels for the crossings in the order they are encountered as walking through the knotoid from one end to the other. Now draw arcs between the marks on $S$ that have the same label. Orient each arc from the over crossing to the under crossing in the knotoid.  Mark each arc with $+$ or $-$ according to handedness' rule, see Figure \ref{fig:lineGdiag}.

 \begin{figure}[H]
 \centering
\includegraphics[width=0.36\textwidth]{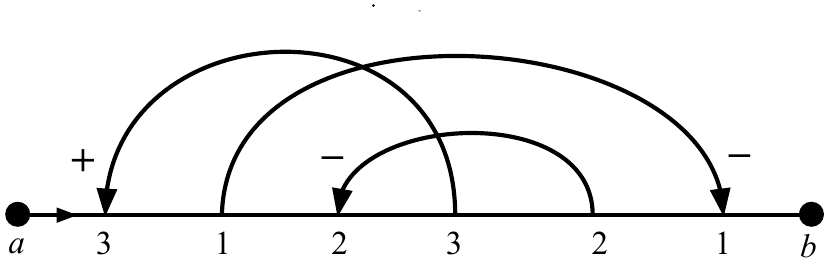}
    \caption{The Gauss diagram arising from knotoid given in Figure \ref{fig:linediag}.} \label{fig:lineGdiag}
\end{figure}

\begin{lemma}\label{lem:gauss} A knotoid diagram is determined by its Gauss diagram (up to a planar isotopy).
\end{lemma}

\begin{proof}  A proof might follow by slightly adapting the classic procedure showing that Gauss diagrams determine knot diagrams given in \cite{Kauff}. Nevertheless, we propose a new approach.  We will see that any face of the knotoid diagram can be determined by using the Gauss diagram. We shall show that the set of faces (and whether two of them share an edge) in the knotoid diagram corresponds to the set of {\em valid} travels in the Gauss diagram.
\smallskip

Let us consider a walk $W$ on the arcs on the boundary of a face (say, counterclockwise) in the knotoid diagram. Notice that we can run through an arc either following its direction or in opposite direction, see Figure \ref{fig:walk}.

 \begin{figure}[H]
 \centering
\includegraphics[width=0.23\textwidth]{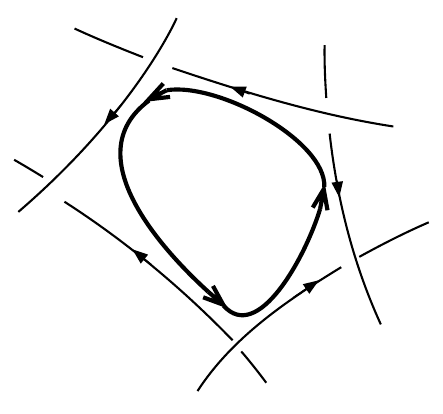}
    \caption{Walk around a boundary of a face.} \label{fig:walk}
\end{figure}

We associate to $W$ a travel $T$ along the Gauss diagram. Turns of $W$ at a crossing correspond to jumps in arrows (from head/tail to tail/head) and each arc in $W$ corresponds to a piece of the oriented segment $S$ (of the Gauss diagram). Travel $T$ may follow or not the direction of $S$  (depending on whether or not $W$ follows the direction of the corresponding arc). Since $W$ goes around a face then this correspondence obeys certain rules according to the sign of the crossing and from which strand (over/under) the walk comes from and goes to. Theses rules are illustrated in Figure \ref{fig:walkTrial}.

 \begin{figure}[H]
 \centering
\includegraphics[width=0.75\textwidth]{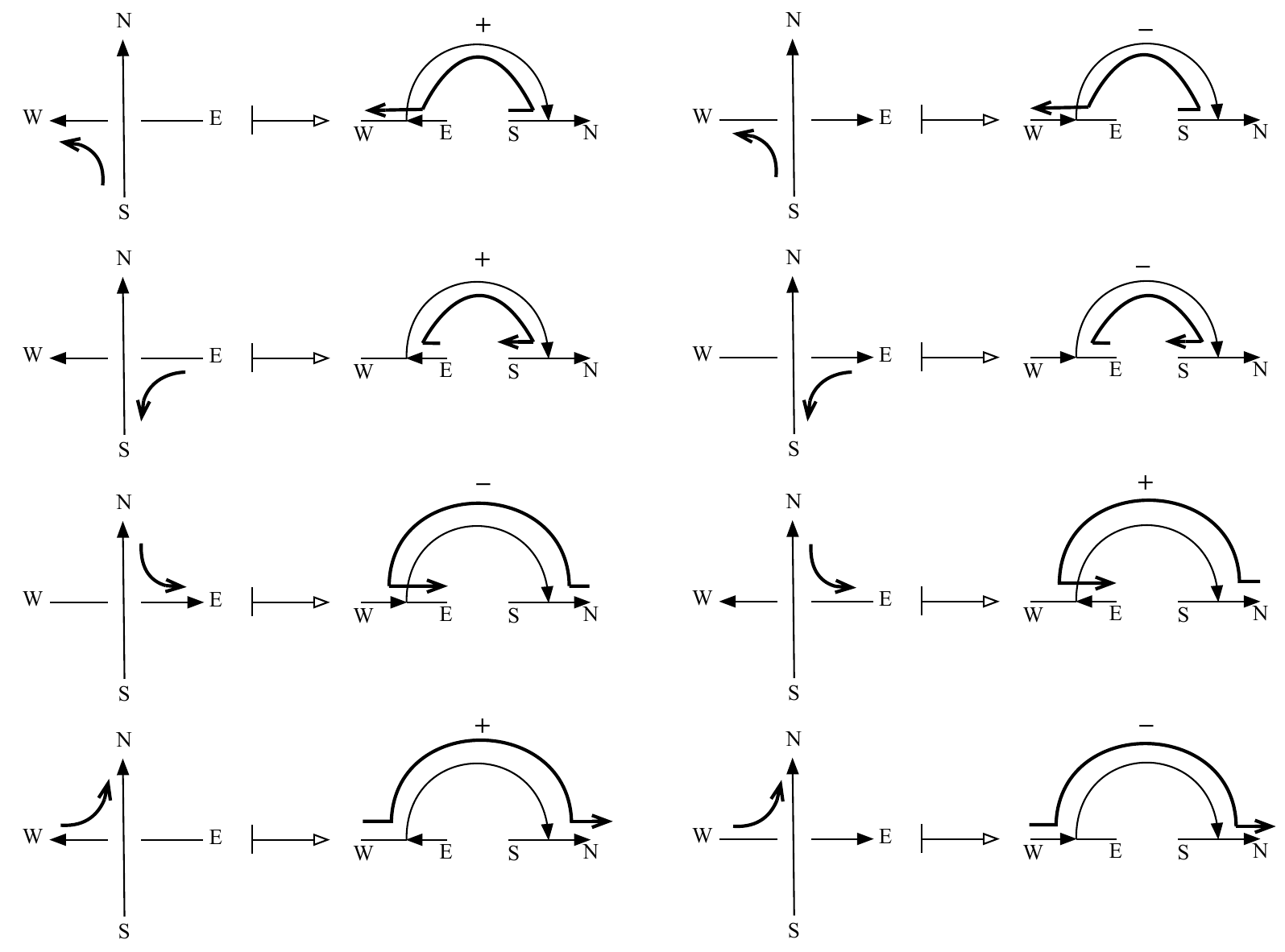}
    \caption{Local turnings of a walk around a crossing and the corresponding local travel in the Gauss diagram (in bold).} \label{fig:walkTrial}
\end{figure}

The travel $T$ constructed this way is called a {\em valid} travel in the Gauss diagram. In fact, if the labels {\bf N}orth, {\bf S}outh, {\bf E}ast and {\bf W}est, as well as the signs of each arrow in a Gauss diagram are fixed  then, a valid travel $T$ verifies the rule : $${\bf N} \rightarrow {\bf  E} \rightarrow {\bf  S}\rightarrow {\bf  W}\rightarrow {\bf  N}$$
where ${\bf X} \rightarrow {\bf  Y}$ means \guillemotleft  go from ${\bf X}$ to ${\bf  Y}$\guillemotright.
 
These rules on $T$ force, by construction, that the corresponding walk $W$ always turns \guillemotleft left\guillemotright \ at each crossing, see Figure \ref{fig:valid}

 \begin{figure}[H]
 \centering
\includegraphics[width=0.3\textwidth]{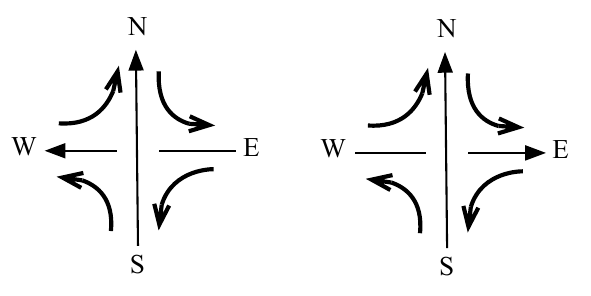}
    \caption{Turns of $W$ induced by valid travels.} \label{fig:valid}
\end{figure}

Therefore, $W$ will be necessarily a walk of a face in the knotoid. Moreover, two valid travels sharing a same piece of $S$ (with opposite direction) correspond to two faces sharing an edge.

Finally, we observe that a valid travel going through an end of the line goes back with opposite orientation, see Figure \ref{fig:walkend}.

 \begin{figure}[H]
 \centering
\includegraphics[width=0.45\textwidth]{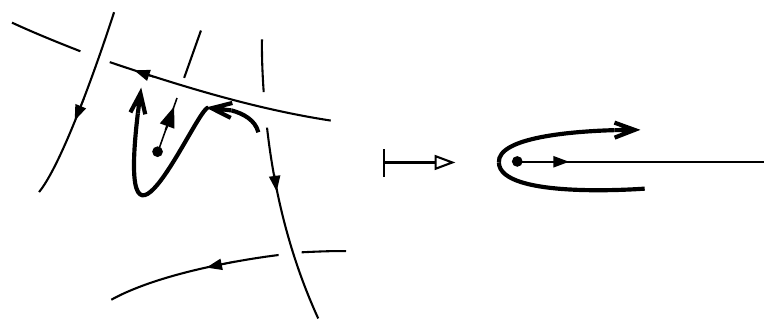}
    \caption{Walk of a face containing a degree one vertex and the corresponding travel.} \label{fig:walkend}
\end{figure}

\end{proof}

\subsection{Main result} Let $y_0 \in \mathbb{R}^3$ and let $\pi_{y_0}$ be the radial projection emitting from $y_0$ to $\stw$, that is, 
$$\begin{array}{lclc}
\pi_{y_0} : & \rth\setminus\{y_0\}& \rightarrow & \stw\\
& y &\mapsto & \frac{y-y_0}{||y-y_0||}
\end{array}$$

Let $K$ be a knot in $\rth$ (not passing through $y_0$). We may associate to $K$ a {\em sphere shadow} which is just the radial projection $\pi_{y_0}(K)$.
We may suppose that such a shadow is regular in the sense that it avoids cusps and tangency points (this can be obtained my making some suitable local modifications to $K$ without changing its type).  Let $z$ be a point in $\pi_{y_0}(K)$, suppose that 
$$z=\pi_{y_0}(y_{1})=\cdots =\pi_{y_0}(y_{k}), \ k\ge 1.$$

We say that $z$ is a {\em simple intersection} point if $k=2$ and a {\em multiple intersection} point if $k\ge 3$. We may avoid multiple intersection points by moving pieces of the shadow around $z$ properly.  As for standard knots, doing such a perturbation does not change the knot type, however, in our case we need slightly more. We have to make sure that the strong geometry captures the information determining the possible diagrams.

The set of possible perturbations is completely determined by the local type of the multiple intersection, that is, the order (say counterclockwise) in which the strands appear as we rotate around the intersection point (because this determines the local diagram up to planar isotopy).  Suppose that the strands are oriented inducing a `right' and a `left' side of the strand. Hence, each strand $a_i$ partitions the other $a_j$, $j\neq i$ into two sets, those crossing $a_i$ left-to-right and those crossing it right-to-left.

\begin{figure}[H]
 \centering
\includegraphics[width=0.5\textwidth]{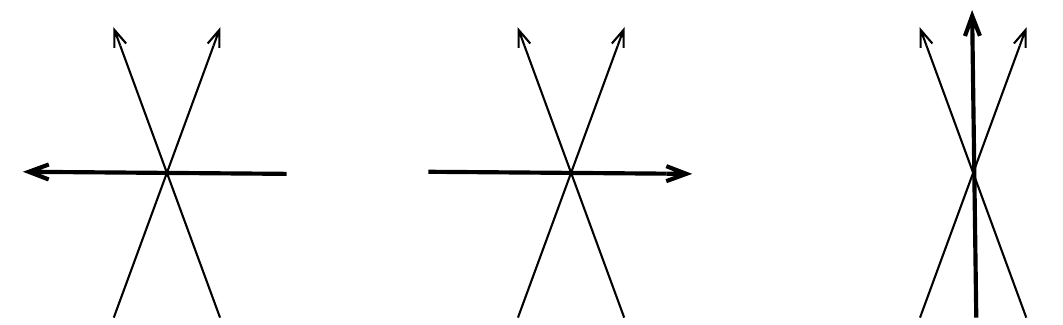}
\caption{(Left) Two consecutive strands (bold strand crosses both  right-to-left) (Center) Two consecutive strands (bold strand crosses both  left-to-right)
(Right) Two nonconsecutive strands (bold strand crosses left-to-right one of them and right-to-left the other).} \label{fig:mod}
\end{figure}

We have that two strands are consecutive if and only they partition the other strands in equal or opposite sets. Therefore, we simply need to know the sign of each crossing in order to know what all the possible local perturbation. This piece of information will be shown to be induced by the strong geometry. 

If the reader is not comfortable with these local permutations, instead, a notion of knot diagram allowing multiple crossings could be defined. The reasoning to check that the strong geometry gives the information about the strands around such multiple crossings would be essentially the same as above.

Therefore, sphere shadows can be thought of as  4-regular maps (i.e., an embedding of a 4-regular planar graph into $\stw$). A {\em spherical diagram} of $K$ is obtained from such a shadow by endowing the under/over information to each vertex, see Figure \ref{fig:tref}.

 \begin{figure}[H]
 \centering
\includegraphics[width=0.21\textwidth]{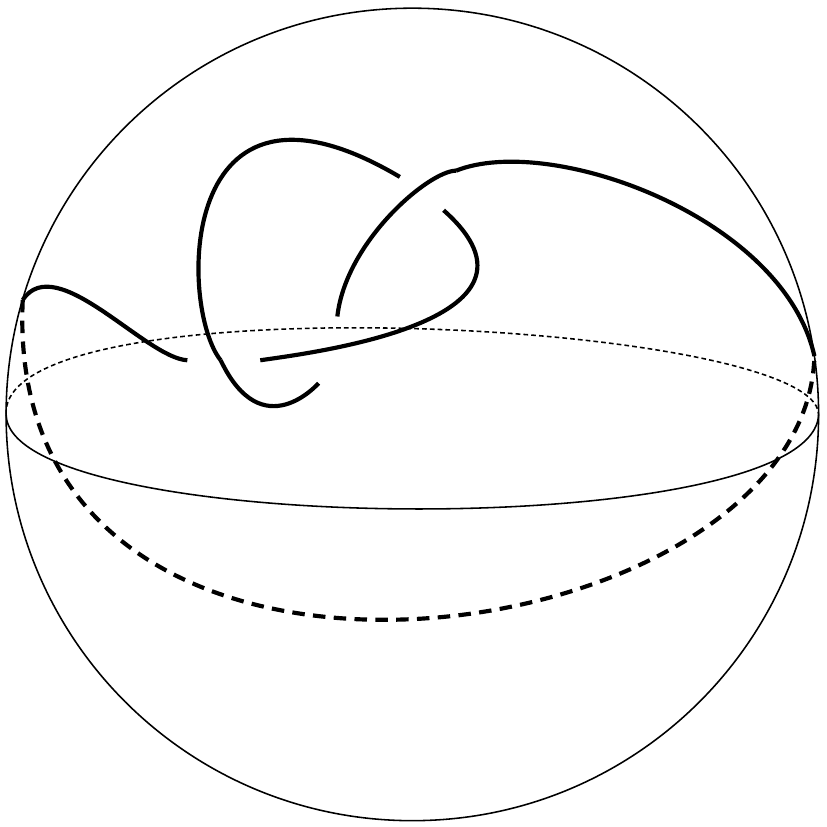}
    \caption{A spherical diagram of the Trefoil.} \label{fig:tref}
\end{figure}

A spherical diagram is still a knot diagram of $K$ in the usual sense. In particular, knots having isotopic spherical diagrams are isotopic.
\smallskip

We may now prove Theorem \ref{thm:main}.
\smallskip

{\em Proof of Theorem \ref{thm:main}.} Let $X=(x_0,\dots ,x_{n-1})$ be tuple of points in $\mathbb{R}^{3}$ in general position. Let $K$ be the polygonal knot formed by segments $[x_0,x_1]\cup [x_1,x_2]\cup\cdots\cup [x_{n-1},x_0]$. We show that $\sgeom(X)$ determines uniquely the type of $K$.  Let us consider the radial projection $\pi_{x_0}(K - \{x_0\})$ of $K\- \{x_0\}$ emitted from $x_0$.  

\begin{remark}\label{rem:proj} (a) $\pi_{x_0}$ maps each semi-open interval $(x_0,x_1]$ and $(x_0,x_{n-1}]$ into a single point. We thus have that the corresponding spherical diagram induced by $\pi_{x_0}(K)$ is a knotoid, denoted by $\mathcal K$.

(b) Since $K$ is polygonal with points in generic position then the shadow $\pi_{x_0}(K- \{x_0\})$ has no tangency point.

(c) $\pi_{x_0}(K)$ may have multiple intersection points. These can be fixed/modified as explained above.
\end{remark}

We notice that any planar isotopy between knotoids can be extended to an isotopy of knots by fixing $x_0$ and keeping the segments $[x_{n-1},x_0]$ and $[x_0,x_1]$ straight while doing the isotopy.

Therefore, by Lemma \ref{lem:gauss}, it is enough to show that the Gauss diagram associated to knotoid $\mathcal K$ is completely determined by $\sgeom_{\mathsf{Aff}}(X)=(M_{\mathsf{Aff}}(X), M_\wed(\inc(X)))$ (implying that $K$ can also be uniquely reconstructed from such strong geometry).
To this end, we show that $\sgeom_{\mathsf{Aff}}(X)$ determines the following two points.
\smallskip

(A) the pairs of arcs that intersect (and which one is over/under the other one) together with the corresponding sign,

(B) the order of the intersections (if more than one) along a given arc in $\mathcal K$. 
\smallskip

For (A), let $\beta=\pi_{x_0}([x_i,x_{i+1}])$ and $\beta'=\pi_{x_0}([x_{i'},x_{i'+1}])$ be two arcs in  $\mathcal K$. Suppose that these arcs intersect in a simple point. 
Then, $\beta$ is over $\beta'$ if the Radon partition of the set $\{x_0,x_i,x_{i+1},x_{i'},x_{i'+1}\}$ is given by $\{x_0,x_{i},x_{i+1}\}\sqcup \{x_{i'},x_{i'+1}\}$, see Figure \ref{fig:rad}.

 \begin{figure}[H]
 \centering
\includegraphics[width=.5\textwidth]{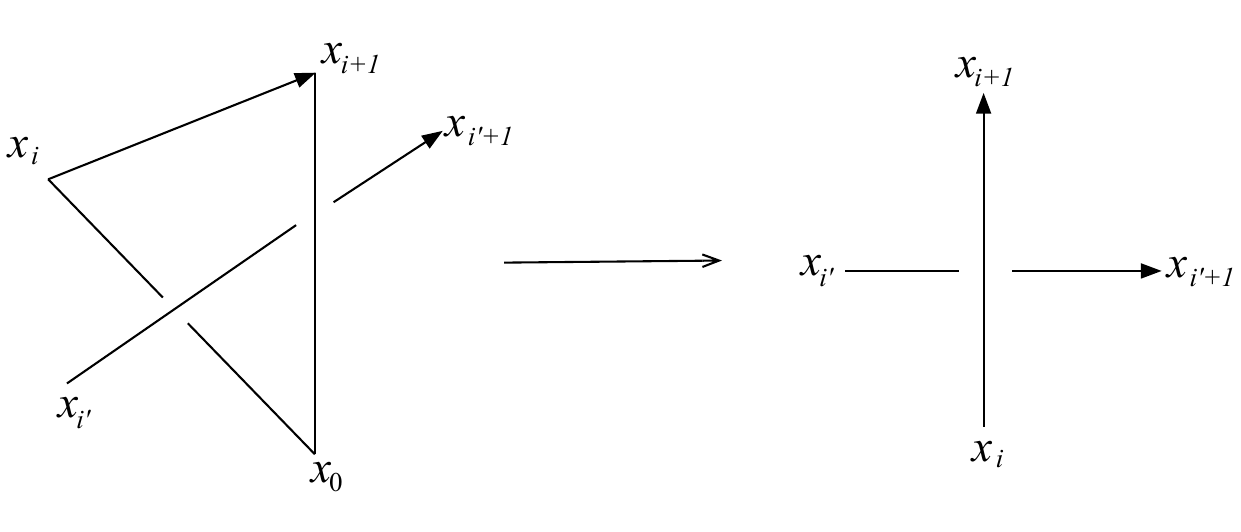}
    \caption{Radon partition and crossing in $\mathcal K$ from the witness $x_0$.} \label{fig:rad}
\end{figure}

Furthermore, if arc $\pi_{x_0}([x_i,x_{i+1}])$ intersect arc $\pi_{x_0}([x_{i'},x_{i'+1}])$ then the oriented arcs (from $x_i$ to $x_{i+1}$ and from $x_{i'}$ to $x_{i'+1}$) on the sphere meet at a positive crossing if and only if $\chi(x_i,x_{i+1},x_0,x_{i'+1})=+1$.
\smallskip

For (B), suppose that both arcs $\beta'=\pi_{x_0}([x_{i'},x_{i'+1}])$ and  $\beta''=\pi_{x_0}([x_{i''},x_{i''+1}])$ intersect arc $\beta = \pi_{x_0}([x_{i},x_{i+1}])$. The order of the intersection, say from $\pi_{x_0}(x_i)$ to $\pi_{x_0}(x_{i+1})$ is given by the chirotope $\chi_{x_0}$ according with the rule given in Figure \ref{fig:wit}. We notice that, by Proposition \ref{prop:wit2}, 
$\chi_{\wedge, x_0}$ can be completely determined from $\chi_{\Lambda}$. 

 \begin{figure}[H]
 \centering
\includegraphics[width=.4\textwidth]{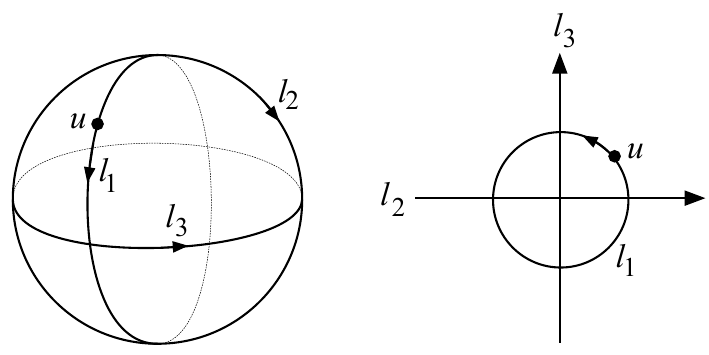}
    \caption{If $\chi_{x_0}(1,2,3)=+1$ then, by walking along $l_1$ following its direction, say from $u$, $l_1$ meets $l_3$ negatively, then $l_2$ positively, then $l_3$ positively and finally $l_2$ negatively.} \label{fig:wit}
\end{figure}

Let $l_1=i\wedge (i+1), l_2=i' \wedge (i'+1)$ and $l_3=i'' \wedge (i''+1)$. Then, by following Figure \ref{fig:wit} with $\chi_{x_0}(1,2,3)= +1$ we have that
$$\text{ if } \beta \text{ crosses }\left\{\begin{array}{l}
\beta' \text{ positively and }\beta'' \text{ positively}  \text{ then }  \beta \text{ meets } \beta'' \text{ before } \beta',\\
\beta' \text{ positively and }\beta'' \text{ negatively}  \text{ then }  \beta \text{ meets } \beta' \text{ before } \beta'',\\
\beta' \text{ negatively and }\beta'' \text{ negatively}  \text{ then }  \beta \text{ meets } \beta'' \text{ before } \beta',\\
\beta' \text{ negatively and }\beta'' \text{ positvely}  \text{ then }  \beta \text{ meets } \beta' \text{ before } \beta''.\\
\end{array}\right.$$
The result follows.
 \hfill$\square$


\section{Spatial graphs}\label{sec:spatgraknot}

A {\em spatial graph} is an embedding of a finite graph in $\rth$, that is, the vertices are distinct points and the edges are simple curves  between them in such a way that any two curves are either disjoint or meet at a common vertex. A spatial graph is called {\em linear} if each edge is a straight line segment.   Two spatial graphs are said to be {\em equivalent} if they are {\em ambiently} isotopic \cite{Kauff1,Yam}. 

\begin{theorem}\label{thm:main2} Let ${R}(G)$ and ${R}'(G)$ be two linear spatial representations of $G$. Let
$X$ and $X'$ be the corresponding set of points in $\rth$ associated to the vertices in such representations. 
If  $\sgeom_{\mathsf{Aff}}(X)$ and $\sgeom_{\mathsf{Aff}}(X')$ are isomorphic then ${R}(G)$ is equivalent to ${R}'(G)$.
\end{theorem}

Let us briefly recall the notion of spatial graphoids needed for the proof of Theorem \ref{thm:main2}. 


A {\em graphoid} is an equivalence class of {\em graphoid diagrams}. A {\em graphoid diagram} is a graph with various distinguished degree-one vertices
generically immersed in $\mathbb{S}^2$, where each double point is decorated as a classical crossing labeled over/under, see Figure \ref{figvirtualgraph}.

 \begin{figure}[H]
 \centering
\includegraphics[width=.25\textwidth]{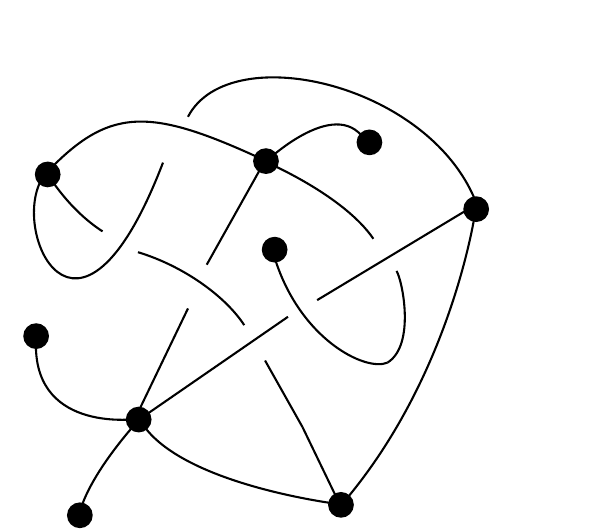}
    \caption{A graphoid with four distinguished degree-one vertices.} \label{figvirtualgraph}
\end{figure}

Graphoids have been studied in different contexts, for instance, in connection with a topological structure that arises in proteins \cite{GKP}. A Gauss code is associated to any graphoid diagram $\mathcal G$ as follows.  We choose an orientation for each edge. We then write down the Gauss code for the shadow by recording whether each arrival at a crossing is an over-crossing (O) or an under-crossing (U), we can also label each crossing by its sign $+$ or $-$ according to handedness’ rule, see Table \ref{tab1}. The difference with usual knots is that we add one more piece of information. For each vertex $v$ of $G$, we consider the sequence of edges incident to $v$, appearing in counterclockwise order around $v$. These {\em edges sequences} are considered up to cyclic permutation.
\medskip

\begin{center}
\captionof{table}{Gauss code of the spatial graphoid given in Figure \ref{figvirtualgraph}.}\label{tab1}
\begin{tabular}{cl}
\smash{\raisebox{-.5\height}{\includegraphics[height=3.9cm]{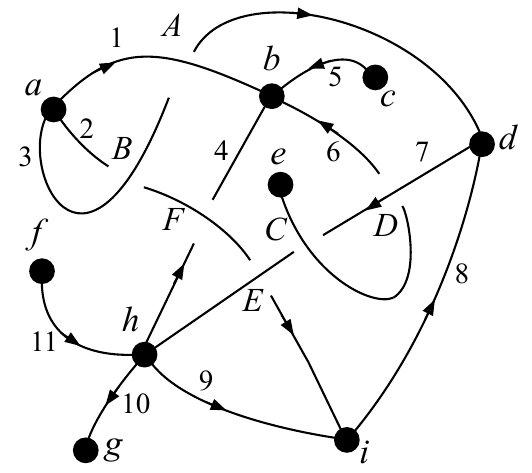}}}&
{\small  \begin{tabular}{l}
\underline{Gauss code}\\
$a  1 A O_+  b$\\
$a 2 B U_- F O_+E U_+ i$\\
$a 3 B O_- A U_+ d$\\
$b 4 F U_+ h$\\
$b 5 c$ \\
$b 6 D U_- C O_- e$ \\
$d 7 D O_- C U_- E O_+ h$\\
$d 8 i$\\
$i 9 h$\\
$h 10 g$\\
$h 11 f$\\
\end{tabular}}
\end{tabular}
\end{center}
\medskip

More conveniently, Gauss codes can be understood as Gauss  \lq arrow\rq\  diagrams (as for knotoids). 
We also consider  edges sequences around each vertex. We label points along the edges according to the crossing sequences of the Gauss code and draw an arrow between each pair of occurrences of the same label. The arrow is oriented from the undercrossing edge to the overcrossing edge and labeled with the sign of the crossing. We notice that once the arrows are drawn, the labels on the edges are redundant, and can thus be removed. In general, there is no canonical choice for \lq the  immersion\rq  of the graph. Even if the underlying graph has a planar embedding, the order of the edges around the vertices may prevent us from using it, see Table \ref{tab2}.
\medskip

\begin{center}
\captionof{table}{Gauss diagram of the graphoid diagram given in Figure \ref{figvirtualgraph}.}\label{tab2}
\begin{tabular}{cl}
\smash{\raisebox{-.5\height}{\includegraphics[height=3.9cm]{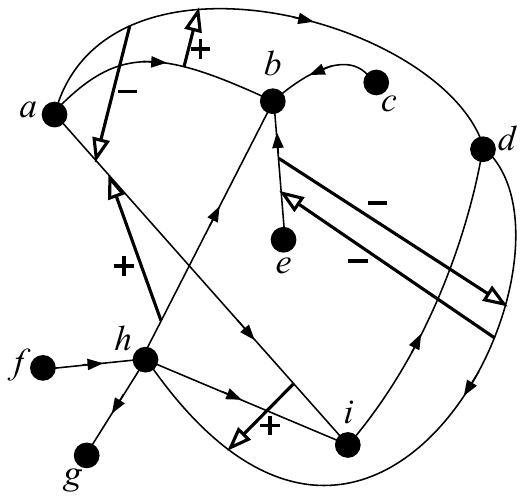}}}&
{\small  \begin{tabular}{l}
\underline{Edge sequences}\\
$a :1,3,2$\\
$b :1, 4, 6, 5$\\
$c : 5$\\
$d :3, 7, 8$\\
$e : 6$ \\
$f : 11$\\
$g : 10$\\
$h :4, 11, 10, 9, 7$\\
$i : 8, 2, 9$
\end{tabular}}
\end{tabular}
\end{center}
\medskip



\begin{lemma}\label{lem:gaussgraphdiag} A graphoid diagram is determined by its Gauss diagram.
\end{lemma}

\begin{proof} We essentially follow the same arguments as those given in the proof of Lemma \ref{lem:gauss}. We only need to adapt the construction of a valid travel $T$ in the Gauss diagram of a spatial graphoid.  We have to determine the rule when $T$ gets to a vertex of degree $\geq 3$ (not existing in knotoids), that is, we have to determine the turns of the corresponding walk $W$ when going through vertices in $\mathcal G$. For this, we use the edge sequence information, if $W$ arrives at $v$ along an edge $e$ then it leaves $v$ by following the edge $f$ appearing before $e$ in counterclockwise direction, that is, the next edge in clockwise direction. This new rule forces that $W$ always turns \guillemotleft left\guillemotright \ at each vertex staying in the same face, see Figure \ref{figPTgraph}.

 \begin{figure}[H]
 \centering
\includegraphics[width=.21\textwidth]{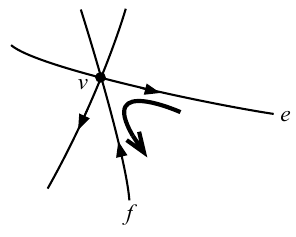}
    \caption{Turn of $W$ at vertex $v$.} \label{figPTgraph}
\end{figure}

\end{proof}

\begin{remark} Lemma \ref{lem:gaussgraphdiag} can be naturally extended to {\em virtual} graphoids, that is, graphoids having {\em virtual crossings}. The virtual notion was introduced by Kauffman in his seminal paper \cite{Kauff}. The idea is not that a virtual crossing is just an ordinary degree-four vertex but rather that a virtual crossing is not really there (hence the name \lq virtual\rq). 
\end{remark}

We may now prove Theorem \ref{thm:main2}.
\smallskip

{\em Proof of Theorem \ref{thm:main2}.} We shall mimic the proof of Theorem \ref{thm:main}. 
Let $X=(x_0,\dots ,x_{n-1})$ be the tuple of points in $\mathbb{R}^{3}$ associated to the vertices $(v_0,\dots ,v_{n-1})$ in the spatial representation of $R(G)$. Let $\pi_{x_0}(R(G))$ be the radial projection of $R(G)$ emitted from $x_0$. It can be checked that such projection induces a graphoid, say $\mathcal G$, with as many distinguished vertices as the degree of the vertex corresponding to point $x_0$ (this can be justified in a similar way as for knotoids arising from radial projections of polygonal knots). 
If the diagram is not connected, then the corresponding components are unlinked, and a sequence of isotopies for each component leads to an isotopy for the whole spatial graph. Therefore, we can suppose without loss of generality that the diagram is connected. This implies that the faces of the diagram have connected boundaries and thus the arguments used in the proof of Theorem \ref{thm:main} work similarly in this case.
\smallskip

Let $\mathcal{G}'$ be the spatial graphoid arising from $R'(G)$. We notice that any planar isotopy between the spatial graphoids  $\mathcal G$ and $\mathcal{G}'$ can be extended to an isotopy between $R(G)$ and $R'(G)$ by fixing $x_0$ and keeping the segments $[x_0,x_i]$, for each $x_i$ where its associated vertex $v_i$ is a neighbor of $v_0$, straight while doing the isotopy.
Therefore, by Lemma \ref{lem:gaussgraphdiag}, it is enough to show that the Gauss diagram associated to $\mathcal G$ is completely determined by $\sgeom_{\mathsf{Aff}}(X)=(M_{\mathsf{Aff}}(X), M_\wed(\inc(X)))$ (implying that $R(G)$ can also be uniquely reconstructed from such strong geometry). 
\smallskip

As discussed in the proof of Theorem \ref{thm:main}, $\sgeom(X)_{\mathsf{Aff}}$ determines in the diagram : the pairs of intersecting arcs (and which one is over/under the other one) together with the corresponding sign, the local type of multiple intersection and thus their local modifications and the order of the intersections (if more than one) along a given arc. Although all of these are needed to determine $\mathcal G$, we still have to show that $\sgeom(X)_{\mathsf{Aff}}$ also determines the order of appearance, say counter clockwise, of the edges around each vertex. In fact this order is determined by $M_{\mathsf{Aff}}(X)$ itself. Indeed, the two pieces of oriented equators $\pi_{x_0}([x_i, x_j])$, $\pi_{x_0}([x_i, x_{j'}])$ meet either positively or negatively depending on the sign of $\Delta( \pi_{x_0}(x_i),\pi_{x_0}(x_j), \pi_{x_0}(x_{j'})) = \chi(x_0, x_i, x_j, x_{j'})$. But the data of the signs of such crossings between pieces of equators is equivalent to the order at which the $x_j$'s appear around $x_i$ (on the unit sphere centered on $x_0$). 
This is due to the fact that the cyclic order of the strands is given by the sign of each pair (as explained above) and, by Proposition \ref{prop:wit1}, the sign of a pair $(x_i \wedge x_j), (x_i \wedge x_{j'})$ is given by the sign $\chi(x_i, x_j, x_{j'})$. Essentially, the circular order around $x_i$ is the `twofold witnessed' (first by $x_0$ and then by $x_i$) rank $2$ chirotope.
\hfill$\square$
\smallskip

We end with the following

\begin{question} Let $R(G)$ be a linear spatial representations of $G$. Let $X$ be the set of points in $\rth$ associated to the vertices of such representation. 
Is it true that $R(G)$ is determined by $M_{\mathsf{Aff}}(X)$ ?
\end{question}

\section{Further research directions}\label{sec:discuss}
In the process of this work, a great deal on the structure, properties and extensions of the notion of strong geometry have been revealed. 

\subsection{Universality}
Ringel’s isotopy conjecture asked whether any two arrangements of lines with the same topological type can be isotoped into each other. In other words, Ringel's conjecture asked whether the realization space of rank 3 matroids is connected.  The famous Mn\"ev's Universality Theorem \cite{M} provides a spectacular negative solution, it states 
that every semi-algebraic variety is stably equivalent to the realization space of some oriented matroid, that is, the realization space of oriented matroids can have the topology of any
semi-algebraic set and, in particular, it can be disconnected.  As mentioned above, oriented matroids do not capture the combinatorics of the cell configuration induced by the spanned lines contrary to strong geometries that completely does. In the same flavor as for oriented matroids, an `inclusion-type' generalization of the universality result \`a la Mn\"ev for rank 3 strong geometries is investigated \cite{G}. 

\subsection{$k$-equivalence}
The notion of knotoid is needed to deal with the fact that the radial projection is from a point belonging to the polygonal knot. This might be avoided if given two set of points $P$ and $Q$ having the same strong geometry then one can show the existence of two points $p$ and $q$ such that $P\cup\{p\}$ and  $Q\cup\{q\}$ also have the same strong geometry. In this case, we just can add this additional point from which the radial projection is applied.  Figuring out the existence of these new points is not difficult when $d=2$ but it is not straightforward for $d\ge 3$. 
\smallskip

It turns out that the notion of strong geometry can be generalized in the planar case as follows. We consider the lines spanned by both the set of the original points and the points arising from line intersections. The latter induce a new set of lines and therefore new intersection points from which new lines can be spanned and so on. These generalizations happen to be closely related to the notion of {\em $k$-equivalence} between two configurations of points. A notion of {\em $\infty$-equivalence} also naturally emanates. Such equivalences are recently studied in \cite{GP} in connection with some asymptotic rigidity properties. The case $k=2$ is related to the existence of a new point to be added to a strong geometry, as discussed above.  

All these investigations require further (much technical) extra work (in progress).

\subsection{Adjoints}
Figure \ref{fig:tran} gives an exemple of an oriented matroid of rank 3 with two realizations that generate non-isomorphic adjoints. Moreover, there exists examples of two non-isomorphic simple oriented matroids of rank 3 that have isomorphic adjoints, see \cite{BK,BK1} and \cite[Exercice 5.13 (a)]{BMS}. In the realizable case, Corollary \ref{cor:chir} asserts that if the adjoints are isomorphic then the oriented matroids are equal or opposite to each other.
\bigskip

{\bf Aknowledgement} The first author was supported by grant PRIM80-CNRS.


\end{document}